\documentclass{amsart}
\usepackage{amssymb,amsmath,graphicx,mathtools}
\usepackage[latin1]{inputenc}
\usepackage[normalem]{ulem}

\newtheorem{thm}{Theorem}[section]
\newtheorem{cor}[thm]{Corollary}
\newtheorem{lem}[thm]{Lemma}

\newtheorem{prop}[thm]{Proposition}
\theoremstyle{definition}
\newtheorem{defn}[thm]{Definition}

\theoremstyle{remark}
\newtheorem{rem}[thm]{Remark}

\numberwithin{equation}{section}

\newcommand{\N}{\mathbf{N}}

\newcommand{\C}{\mathbf{C}}

  \newcommand{\tr}[1]{\ensuremath{\prescript{t}{}{#1}}}%
  \newcommand{\cC}{\ensuremath{\mathcal{C} }}%
  \newcommand{\cU}{\ensuremath{\mathcal{U} }}%
  \newcommand{\cF}{\ensuremath{\mathcal{F} }}%
  \newcommand{\cG}{\ensuremath{\mathcal{G} }}%
  \newcommand{\cE}{\ensuremath{\mathcal{E} }}%
  \newcommand{\cR}{\ensuremath{\mathcal{R} }}%

\makeatletter
\let\@wraptoccontribs\wraptoccontribs
\makeatother

\begin{document}

\title{A local characterization of Kazhdan projections and applications}

\date{}
\author[De La Salle]{Mikael De La Salle} \address{UMPA, CNRS ENS de Lyon\\ Lyon\\FRANCE}\email{mikael.de.la.salle@ens-lyon.fr}

\begin{abstract} 
  We give a local characterization of the existence of Kazhdan projections for arbitary families of Banach space representations of a compactly generated locally compact group $G$. We also define and study a natural generalization of the Fell topology to arbitrary Banach space representations of a locally compact group. We give several applications in terms of stability of rigidity under perturbations. Among them, we show a Banach-space version of the Delorme--Guichardet theorem stating that property (T) and (FH) are equivalent for $\sigma$-compact locally compact groups. Another kind of applications is that many forms of Banach strong property (T) are open in the space of marked groups, and more generally every group with such a property is a quotient of a compactly presented group with the same property. We also investigate the notions of central and non central Kazhdan projections, and present examples of non central Kazhdan projections coming from hyperbolic groups.
\end{abstract}
\maketitle

\section{Introduction}

Let $G$ be a finitely generated group with finite symmetric generating set $S$ and associated word-length $\ell$. Consider the combinatorial laplacian $\Delta$
\emph{i.e.} the element of the group algebra of $G$ defined by
\[ \Delta = \frac{1}{2|S|} \sum_{s \in S} (s-1)^* (s-1) = 1 - \frac{1}{|S|} \sum_{s \in S} s \in \C[G].\]
A unitary representation $(\pi,\mathcal H)$ has spectral gap if there is $\varepsilon$ such that the spectrum of $\pi(\Delta)$  is contained in $\{0\} \cup [\varepsilon,2]$. Since for $t \in [0,2]$, the inequality $(1-\frac t 2)^2 t \leq (1- \frac \varepsilon 2)^2 t$ holds if and only if $t \in \{0\} \cup [\varepsilon,2]$, this is equivalent to the validity of 
\[ \pi\left( (1-\frac 1 2 \Delta)^2\Delta \right)\leq (1-\frac 1 2 \varepsilon)^2 \pi(\Delta) \textrm{ in }B(\mathcal H).\]
If we write $m= 1-\frac 1 2 \Delta = \frac 1 2 + \frac{1}{2|S|} \sum_{s \in S} s$, this is equivalent to the inequality
\begin{equation}\label{eq:local_shrinkingHilbert} \left(\sum_{s \in S} \|\pi(s) \pi(m) x - \pi(m)x\|^2 \right)^{1/2} \leq (1- \frac \varepsilon 2) \left(\sum_{s \in S} \|\pi(s) x - x\|^2 \right)^{1/2}.\end{equation}
for every $x \in \mathcal H$. In words, averaging on the orbit of $x$ with respect to the probability measure $m$ gives a vector which is moved $(1-\frac \varepsilon 2)$ times less than $x$ by the elements of $S$. The validity of \eqref{eq:local_shrinkingHilbert} for every unitary representation $(\pi,\mathcal H)$ of $G$ and $x$ in $\mathcal H$ therefore characterizes when $G$ has a uniform spectral gap for every unitary representation, \emph{i.e.} when $G$ has Kazhdan's property (T). As was already observed in \cite{fishermargulis}, the importance of this criterion for property (T) is that it is local~: if $(\pi,E)$ is a Banach space representation (or more generally an affine action) which is close to a unitary representation (or an isometric action on a Hilbert space) of $G$ then the same inequality will hold with $\varepsilon$ replaced by $\varepsilon/2$ for $(\pi,E)$. We will make the term ``close to a unitary representation'' precise later, but for example this includes representations on a space close to a Hilbert space (in the sense that the parallelogram identity holds up to a small multiplicative error) of a group close to $G$ in the space of marked groups and such that $\|\pi(g)\|$ is close to $1$ for all $g$ in $S$. This observation proves at the same time (1) Fisher and Margulis's theorem \cite{baderfurmangelandermonod} that every isometric action of a group with property (T) on an $L^p$ space has a fixed point for $p$ close enough to $2$ (2) Shalom's Theorem \cite{shalom} that property (T) is an open property in the space of marked groups (3) the fact that property (T) implies robust property (T) for spaces close to Hilbert spaces. This last fact answers a question raised in a preliminary version of \cite{oppenheim}. 

All the preceding was probably known to experts, and in particular to the authors of \cite{fishermargulis}. The first original contribution of this work is that the existence of a measure satisfying \eqref{eq:local_shrinkingHilbert} characterizes the existence of so called Kazhdan projection, not only for unitary representations as the short computation above proves, but also for arbitrary families of representations of $G$ on Banach spaces, not necessarily isometric or uniformly bounded.

The setting is the following. Let $\cF$ be a collection of representations of $G$ on Banach spaces satisfying the very mild condition
\begin{equation}\label{eq:unif_bound} \sup_{(\pi,E) \in \cF} \|\pi(g)\|_{B(E)} <\infty\ \  \forall g \in S.\end{equation}
This condition allows to define a seminorm on the group algebra $\C[G]$ by 
\begin{equation}\label{eq:def_CF} \|a\|_{\cF} = \sup_{(\pi,E) \in \cF} \|\pi(a)\|_{B(E)}.\end{equation}
We denote by $\cC_{\cF}(G)$ the completion of $\C[G]$ for this seminorm. This is a generalization of the maximal $C^*$-algebra of a group, which corresponds to the case when $\cF$ is the unitary representations of $G$ on a Hilbert space.
\begin{defn} A Kazhdan projection in $\cC_{\cF}(G)$ is an idempotent $p$ belonging to the closure of $\{m \in \C[G], \sum_{g \in G} m(g)=1\}$ such that, for every $(\pi,E)$ in $\cF$, $\pi(p)$ is a projection on the space of invariant vectors $E^\pi=\{x \in E, \pi(g) x=x \forall g \in G\}$ .

A Kazhdan projection is called central if it belongs to the center of $\cC_{\cF}(G)$.
\end{defn}
The importance of such projections for general Banach space representations comes from the work of Lafforgue \cite{lafforguestrongt}, see subsequent work \cite{lafforguefastfourier,liao,oppenheim,delasalle1,delaatdelasalle1}. Kazhdan projections have been studied in depth recently in \cite{drutunowak}, but our terminology is a bit different, as they call a Kazhdan projection what we call here a central Kazhdan projection. Motivations for this choice of terminology are presented in \S \ref{sec:hamlet}, where the disctinction between central and non central Kazhdan projections is made clear. In particular in Corollary \ref{cor:KP_automatically_central} it is proved that when $\cC_\cF(G)$ is stable by duality, then a Kazhdan projection is always central. We also present in Remark \ref{rem:noncentralKP} some natural examples where there are Kazhdan projections but no central Kazhdan projections. See Remark \ref{rem:comment_K_proj} for other comments.

If $(\pi,E)$ is a representation of $G$ (or more generally an action of $G$ on a Banach space $E$) and $x \in E$, we will measure by
\begin{equation}\label{eq:defdelta} \delta_S^\pi(x) = \max_{s \in S} \|\pi(s) x - x\|_E\end{equation}
the maximal amount by which $x$ is moved by the elements of $S$. We could as well have defined $\delta_S^\pi$ by the formula $(\sum_{s \in S} \|\pi(s) x - x\|^2 )^{1/2}$ as in \eqref{eq:local_shrinkingHilbert}, but since we do not only work with Hilbert spaces, we prefer to use formula \eqref{eq:defdelta}, which is not less relevant but is simpler.

Our main new contribution is the following local characterization of Kazhdan projections, which generalizes to non unitary representations the easy observation from the beginning of the introduction.
\begin{thm}\label{thm=local_characterization_of_Kazhdan_constant} $\cC_{\cF}(G)$ contains a Kazhdan projection if and only if there exists $m \in \C[G]$ with $\sum_g m(g)=1$ such that 
\begin{equation}\label{eq:m_locally_shrinks} \delta_S^\pi(\pi(m) x) \leq \frac 1 2\delta_S^\pi(x) \textrm{ for all }(\pi,E)\in \cF\textrm{ and }x \in E.\end{equation}
If these properties hold and if $\sigma$ is an affine action of $G$ whose linear part belongs to $\cF$, then $\sigma$ has a fixed point if and only if $\delta_S^\sigma(\sigma(m) x) \leq \frac 1 2\delta_S^\sigma(x)$ for all $x \in E$.
\end{thm}
We point out the following~: contrary to the previous characterizations of Kazhdan projections in \cite{drutunowak}, the fact that $E^\pi$ has a complement subspace is not part of the hypothesis, it is a consequence of \eqref{eq:m_locally_shrinks}. Of course, Theorem \ref{thm=local_characterization_of_Kazhdan_constant} remains true if $\frac 1 2$ is replaced by any number in $(0,1)$.

The main interest of this characterization of the existence of Kazhdan projections is that it is completely local~: if the support of $m$ is contained in $B_R=\{g \in G, |g|_S \leq R\}$, and if $(\pi,E)$ is a representation of $G$ such that for every $x \in E$, the $B_{R+1}$-orbit of $x$ is ``almost isometric'' to the $B_{R+1}$-orbit of a point $x' \in E'$ for a representation $(\pi',E') \in \cF$, then (\ref{item=local_shrinking}) also holds for $(\pi,E)$, perhaps with $\frac 1 2$ replaced by $\frac 1 2+\varepsilon$. This opens the possibility of applying ultraproduct constructions as in \cite{gromov,fishermargulis}. This simple observation by itself is however not very useful from a representation-theoretical point of view, as this notion of $(\pi,E)$ being close to a representation in $\cF$ is strong. On the opposite there is a natural topology (related to the Fell topology) on every set of Banach space representations of $G$ satisfying \eqref{eq:unif_bound}, that we define and study in \S \ref{sec:topology}. To feel the difference between these two notions of representations being close, consider for example the case when $x_n$ is a sequence of almost invariant unit vectors of a unitary representation $\pi$. For the Fell topology, this says that $\pi$ is close to the trivial representation, which does not say anything about the validity of \eqref{eq:m_locally_shrinks} for $x_n$. To obtain useful information, it is better to zoom around $x_n$. Then the origin disappears from the vision, and the representation now looks much more like an affine action with linear part a representation close to $\pi$ in the Fell topology. This vague discussion should give an informal explanation why cohomology enters into our second main result, and this will be made precise in its proof. Recall that $H^1(G;\pi)=0$ means that every affine action of $G$ with linear part $\pi$ has a fixed point. See Theorem \ref{thm:Kazproj} for a more precise statement.
\begin{thm}\label{thm:Kazproj+H1_open} Let $\cF$ be a set of representations of $G$ satisfying \eqref{eq:unif_bound}. If $\cC_{\cF}(G)$ has a Kazhdan projection and $H^1(G;\pi) = 0$ for every $(\pi,E) \in F$, then there is a strong neighbourhood (see Definition \ref{defn:strong_neigh}) $\cF'$ of $\cF$ such that  $\cC_{\cF'}(G)$ has a Kazhdan projection and $H^1(G;\pi) = 0$ for every $(\pi,E) \in \cF'$.
\end{thm}

We end the introduction by listing several consequences of this result, which follow from the understanding, obtained in \S \ref{subsection:strongneig_ex}, of strong neighbourhoods in several examples. 

If $\cE$ is a collection of Banach spaces and $m$ is a function from $G$ to $(0,\infty)$, we denote by $\cF(\cE,m)$ the collection 
 of all representations $(\pi,E)$ on a space $E \in \cE$ and such that $\|\pi(g)\|_{B(E)}\leq e^{m(g)}$ for all $g$ (if $m=c\ell$ this boils down to the inequality $\max_{g \in S} \|\pi(g)\| \leq e^c$).

Here are some of the consequences. Definitions can be found in the body of the paper. Precise statements and other results can be found in \S \ref{subsection:app_fixed_point}, \ref{subsec:Lp} and \ref{subsection:compactlypresented}. 
\begin{itemize}
\item(Corollary \ref{cor:T_implies_robustT}) If $G$ has property (T), then there is $\varepsilon>0$ such that $\cC_\cF(G)$ has a central Kazhdan projection and $H^1(G;\pi)=0$ for every $(\pi,E) \in \cF$, where $\cF$ is the collection of all representations $(\pi,E)$ such that $\max_{g \in S} \|\pi(g)\|\leq 1+\varepsilon$ on a Banach spaces satisfying 
\[\frac 1 2  \left(\|x+y\|^2 + \|x-y\|^2 \right) \leq (1+\varepsilon) \left(\|x\|^2 + \|y\|^2 \right) \ \ \forall x,y \in E.\]
With the vocabulary of \cite{oppenheim} or Definition \ref{defn:robustT}, property (T) is equivalent to robust property (T) with respect to Banach spaces satisfying the preceding.
\item(Corollary \ref{cor:FLp_implies_robustFLp}) Let $1<p<\infty$. $G$ has property (F$_{L_p}$) if and only if there is $\varepsilon>0$ such that $G$ has robust property (T)  with respect to $\{L_q, |p-q|<\varepsilon\}$.\footnote{In particular, the set of values of $p \in (1,\infty)$ such that $G$ has (F$_{L_p}$) is open. Although we are not aware of a place where this remark has already been made, we are sure that this was well-known, as the proof by Fisher and Margulis of the case $p=2$ \cite[Lemma 3.1]{baderfurmangelandermonod} applies with almost no change.}
\item (Corollary \ref{cor:Flp_open_property}) Let $\cE$ be a class of superreflexive Banach spaces closed under ultraproducts (for example the class of $L_p$ spaces for some $1<p<\infty$). Then the set of finitely generated groups with property (F$_\cE$) is open in the space of marked groups.
\end{itemize}

The second point above is also valid for many other reasonable classes of Banach spaces, for example non-commutative $L_p$ spaces. See Corollary \ref{cor:Delorme-Guichardet_Banach} and \ref{cor:FE}. This statement has to be compared to a celebrated theorem by Delorme and Guichardet asserting that, for a countable group, property (T) is equivalent to (F$_H$). It is well-known that strictly speaking, the Delorme--Guichardet theorem is no longer true for Banach space and for example $L_p$ spaces for $p$ large (see Remark \ref{rem:noncentralKP}): there are groups which have (F$_{L_p}$) but not (T$_{L_p}$). Corollary  \ref{cor:FLp_implies_robustFLp}, which characterizes (F$_{L_p}$) in terms of the existence of a Kazhdan projection for some class of representations on $L_p$, should be considered as the correct Banach-space analogue of the  Delorme--Guichardet theorem.

The above results, and all other results in the paper, are valid more generally for locally compact compactly generated groups\footnote{There is a small subtlety related to continuity of representations, so the assumption that $\cE$ is stable by ultraproducts has to be replaced by stability by finite representability.}, and can be combined. For example, if $\cE$ is closed under finite representability and $G$ has (F$_\cE$), then there is a compactly presented group $G'$ which surjects on $G$, an integer $N$ and a positive number $\varepsilon>0$ such that $G$ has robust (T) with respect to the Banach space such that all $N$-dimensional subspaces are at distance less than $(1+\varepsilon)$ from a space in $\cE$. Let us finally mention a result (Corollary \ref{cor:strongTcompactly_presented}) which almost says that every group with Lafforgue's strong property (T) with respect to a reasonable class $\cE$ is a quotient of a compactly presented group with strong property (T) with respect to $\cE$.

\subsection*{Comparision with previous work}

It should be noted that a criterion similar to \eqref{eq:local_shrinkingHilbert} for property (T) was already at the heart of the work of Fisher and Margulis \cite{fishermargulis}. They also exploited that it still holds for actions (not necessarily on a Banach space) which are ``close'' to actions by isometries on Hilbert space, and that for such actions, $\pi(m)^nx$ converges to a fixed point of $\pi$, so in particular $\pi$ has a fixed point ``not too far from $x$''. This allowed them to reprove Shalom's theorem that property (T) is open in the space of marked groups. This technique also allowed them to prove that (T) implies (F$_{L_p}$) for $p$ small enough to $2$ (but this was only written in \cite{baderfurmangelandermonod}, and actually without relying on \eqref{eq:local_shrinkingHilbert}, where they allowed a less restrictive meaning of representations being closed). Compared to \cite{fishermargulis}, we make the choice to work only with linear/affine actions on Banach spaces, but we do not restrict to (close to) isometric actions, and we discuss in length the notion of closeness, and its relation to the variant Fell topology that we consider in \S \ref{sec:topology}. The characterization of Theorem \ref{thm=local_characterization_of_Kazhdan_constant} is new.

Theorem \ref{thm:Kazproj+H1_open} can be informally expressed by saying that deforming a representation for this Fell-like topology preserves the vanishing of the first cohomology group. Recently Bader and Nowak \cite{badernowak} also studied how deforming a representation affects its cohomology groups. Our results do not seem to be comparable, since they work with a much stronger notion of deformation than ours, where the Banach space is unchanged and the generators act by operators close in the norm topology from the original generators.

\section*{Acknowledgements} I thank Masato Mimura and Izhar Oppenheim for interesting discussions, and Masato Mimura for allowing me to include his argument for Lemma \ref{lem:FE_implies_uniform_spectral_gap}. I also thank the referee for useful comments.

\tableofcontents

\section{Preliminaries}

\subsection{The group $G$}\label{subsection:groupG}

Throughout this paper $G$ will be a compactly generated locally compact infinite group, and $S \subset G$ will be a compact symmetric generating set. We assume that the identity of $G$ belongs to $S$, so that (Baire) every compact subset of $G$ is contained in $S^N$ for some $N$. We denote by $\ell$ the word length function 
\begin{equation}\label{eq:word-length}
\ell(g) = \min\{n, g \in S^n\}.\end{equation} We also fix a left Haar measure, and we denote $\int f$ or $\int f(g) dg$ the integration with respect to it. 

We will also assume that $G$ is separable. This assumption is just for convience; all the results of the paper remain true if $\kappa$ is the cardinality of a dense subset of $G$ and every occurence of the word separable is replaced ``with a dense subset of cardinality $\leq \kappa$''.

By an approximate unit in $C_c(G)$ we mean a net $f_n \in C_c(G)$ such that $\int f_n=\int |f_n|=1$ and for every neighbourhood $V$ of $e$, the support of $f_n$ is contained in $V$ for all $n$ large enough.

\subsection{Representations and affine actions}
By \emph{representation} of $G$ we will always mean a pair $(\pi,E)$ of a Banach space $E$ and a strongly continuous representation $\pi$ of $G$ on $E$, \emph{i.e.} $\pi$ is group homomorphism from $G$ to the group $\mathrm{GL}(E)$ of bounded invertible operators on $E$ such that $g \mapsto \pi(g) x$ is continuous for every $x \in E$. Two representations $(\pi_1,E_1)$ and $(\pi_2,E_2)$ are said to be equivalent if there is a surjective linear isometry between $E_1$ and $E_2$ which intertwines the actions. We say that $(\pi,E)$ is an isometric representation if $\pi(g)$ is an isometry of $E$ for all $g \in G$. We will keep the word unitary representation for representations by isometries on a Hilbert space.

In \S \ref{sec:topology} and \ref{sec:applications} we will make the effort to explicitely work with sets of representations. This is a small issue because the class of all representations of a group is not a set. A solution is to consider equivalence classes of representations with some bound on the dimension of the Banach space. For example a reasonable set will be the set of equivalence classes of representations on a separable Banach space. This is indeed reasonable because, $G$ being separable, every Banach space representation is a direct limit of separable Banach space representations. We could also bound the dimension by some inaccessible cardinal, which would have the nice feature that our set of Banach spaces will be stable by the operations of duality and most ultraproducts.

The dual (or contragredient) representation of a representation $(\pi,E)$ is the representation $(\tr \pi,\tr E)$ where $\tr E$ is the closed subspace of the  $\xi \in E^*$ such that $g \mapsto \pi(g^{-1})^* \xi$ is continuous, and $\tr \pi(g)$ is the restriction of $\pi(g^{-1})^*$ to this subspace. If $E$ is reflexive then $\tr E = E^*$. In general $\tr E$ is only a weak-* dense subspace of $E^*$, but this implies that $(\pi,E)$ is naturally a subrepresentation of $(\tr{\tr \pi},\tr{\tr E})$ \cite[Lemma 2.3]{delasalle1}.

The space of invariant vectors of a representation $(\pi,E)$ is denoted $E^\pi$:
\[ E^\pi =\{x \in E, \pi(g) x=x \forall g \in G\}.\]
An important fact about isometric representations on reflexive Banach spaces is that the space $E^\pi$ has always a $\pi(G)$-invariant complement subspace \cite{baderrosendalsauer,shulman}. This is not the case for arbitrary Banach space representations. See for example \cite[Remark 2.9]{baderfurmangelandermonod} (respectively \cite{shulman}) where for every non-amenable discrete group, an example of an isometric representation is given where $E^\pi$ has no $\pi(G)$-invariant complement subspace (respectively has no complement subspace at all). There are also examples in a different direction (reflexive spaces but not isometric representations). Indeed, the dual representation of the representation constructed in \cite[Th\'eor\`eme 1.4]{lafforguestrongt} for a hyperbolic group $G$ is a representation with polynomial growth on a Hilbert space where $E^\pi$ has no $\pi(G)$-invariant complemented subspace. In a similar direction, it follows from Remark \ref{rem:noncentralKP} and Proposition \ref{prop:central} that for every hyperbolic group $G$, there exist $1<p<2$ such that, for every $\varepsilon>0$, there is a representation of $G$ on $E=L_p$ such that $\max_{s \in S}\|\pi(g)\| \leq 1+\varepsilon$ and such that $E^\pi$ has a complemented subspace but no $\pi(G)$-invariant complemented subspace.

A continuous affine action of $G$ on a Banach space $E$ is a group homomorphism $\sigma$ from $G$ to the group of continuous invertible affine maps on $E$ such that $g \mapsto \sigma(g) x$ is continuous for every $x \in E$. Since this group is isomorphic to $\mathrm{GL}(E) \ltimes E$, a continuous affine action is of the form $\sigma(g) x=\pi(g) x + b(g)$ for a representation $(\pi,E)$ of $G$ and a continuous function $b \colon G \to E$ satisfying the cocycle relation $b(gh) = b(g)+\pi(g)b(h)$ for all $g,h\in G$. Traditionally, the vector space of such continuous cocycles is denoted by $Z^1(G,\pi)$, the cocycles of the form $b(g) = x-\pi(g) x$ (which correspond exactly to the affine actions with a fixed point) are denoted by $B^1(G,\pi)$, and the quotient vector space $Z^1(G,\pi)/B^1(G,\pi)$ is denoted by $H^1(G,\pi)$. So the formula $H^1(G,\pi)=0$ means that every affine action with linear part $(\pi,E)$ has a fixed point. 

If $(\pi,E)$ is a representation of $G$, and if $m$ is a compactly supported complex measure on $G$, we will denote by $\pi(m) \in B(E)$ the operator $x \in E \mapsto \int \pi(g)x dm(g)$. If $m$ is absolutely continuous with respect to the Haar measure, we will denote $\pi(m)$ by $\pi(f)$ if $f=\frac{dm}{dg}$ is the Radon-Nikodym derivative. We will use the same notation $\sigma(m)x = \int \sigma(g) x dm(g)$ when $\sigma$ is a continuous affine action of $G$ on $X$, and $m$ is a compactly supported measure with $\int 1 dm=1$.

\subsection{Ultrafilters}

An ultrafilter on a set $I$ is a set $\cU$ of subsets of $I$ that is
closed under taking supersets, and such that for every subset $A$ of
$I$, $\mathcal U$ contains either $A$ or $I \setminus A$ (but not
both). As is standard, the set of ultrafilters on $I$ is in natural
bijection with the set of characters of $\ell_\infty(I)$: an
ultrafilter is something that chooses, for every bounded family
$(a_i)_{i \in I}$ of complex numbers, a point in the closure of
$\{a_i, i \in I\}$ in a way compatible with pointwise multiplication
and addition.

If $\cU$ is an ultrafilter on a set $I$, we denote by $(a_i)_{i \in I}
\mapsto \lim_\cU a_i$ the associated character of $\ell_\infty(I)$. It
is characterized by the fact that $A \in \cU$ if and only if $\lim_\cU
1_{i \in A} = 1$.

If $I$ is a directed set, we say that $\cU$ is cofinal if $\lim_\cU
1_{i\geq i_0}=1$ for all $i_0 \in I$. It follows by Zorn's lemma that
cofinal ultrafilters exist on every directed set.

\subsection{Strong neighbourhoods}
In a non-Hausdorff topological space, there is a difference for a net to converge to a point in a subset $A$, and for all its limit points to belong to $A$. The next lemma and the definition that follows, related to this phenomenon, will be important for us.
\begin{lem}\label{lem:characterization_strong_neigh} Let $X$ be a topological space, and $A,B \subset X$. The following are equivalent.
\begin{enumerate}
\item\label{item:strongneig0} For every net $(x_i)_{i \in I}$ in $X$ whose accumulation points are contained in $A$, there is $i_0 \in I$ such that $x_i \in B$ for all $i \geq i_0$.
\item\label{item:strongneig1} For every net $(x_i)_{i \in I}$ in $X$ such that $x_i \in B^c$ for all $i$, the net $(x_i)$ has an accumulation point in $A^c$.
\item\label{item:strongneig2} $B$ belongs to every ultrafilter on $X$ whose accumulation points are contained in $A$.
\end{enumerate}
\end{lem}
\begin{proof} (\ref{item:strongneig0}) $\implies$ (\ref{item:strongneig1}) is obvious.

Assume (\ref{item:strongneig1}). Let $\cU$ be an ultrafilter on $X$ such that $B \notin \cU$. We shall prove that $\cU$ has an accumulation point in $A^c$. Since $\cU$ is an ultrafilter, $B^c \in \cU$, and for every $C \in \cU$, there is $x_C \in C \cap B^c$. By (\ref{item:strongneig1}) the net $(x_C)_{C \in \cU}$ has an accumulation point $x \in A^c$. Then $x$ is an accumulation point of $\cU$ as requested. 

Assume (\ref{item:strongneig2}), and let $(x_i)_{i \in I}$ be a net such that for every $i_0 \in I$, there is $i \geq i_0$ such that $x_i \in B^c$. By Zorn's lemma there is an ultrafilter $\cU$ on $X$ which contains $B^c$ and $\{x_i, i \geq i_0\}$ for all $i$. In particular, $B \notin \cU$, so that by (\ref{item:strongneig2}) $\cU$ has an accumulation point $x \in A^c$. It is in particular an accumulation point of $(x_i)_{i \in I}$, which proves (\ref{item:strongneig0}).
\end{proof}

If $X$ is compact and Hausdorff, then the equivalent properties in Lemma \ref{lem:characterization_strong_neigh} are equivalent to $B$ being a neighbourhood of $A$. This justifies the following definition.
\begin{defn}\label{defn:strong_neigh} If the equivalent properties in Lemma \ref{lem:characterization_strong_neigh} are satisfied, we will say that $B$ is a strong neighbourhood of $A$.
\end{defn}
We warn the reader that a strong neighbourhood of a set $A$ does not necessarily contain $A$. For example, if $X$ contains a point $x_0$ whose only neighbourhood is $X$, then the empty set is a strong neighbourhood of $X \setminus \{x_0\}$. More generally, $x \in A$ belongs to every strong neighbourhood of $A$ if and only if $\overline{\{x\}} \subset A$.

\subsection{Ultraproducts and finite representability}
\label{subsection:ultraproducts}
We recall briefely some facts on the local theory of Banach spaces. We refer to \cite{JohnsonLindenstrauss} (in particular section 9) for a concise introduction. 

If $\cU$ is an ultrafilter on a set $I$ and $E_i,i \in I$ are Banach spaces, we denote by $\prod_\cU E_i$ the ultraproduct Banach spaces, as introduced by Dacunha-Castelle and Krivine \cite{dacunhacastellekrivine}. Recall that $\prod_\cU E_i$ is the quotient of the Banach space $\prod_I E_i$ of bounded families with values in $E_i$ for the norm $\|(x_i)\| = \sup_i \|x_i\|_{E_i}$ by the closed subspace of sequences satisfying $\lim_{\cU} \|x_i\| =0$. The equivalence class of $(x_i)_{i \in I}$ will be denoted by $(x_i)_\cU$. Its norm is $\lim_\cU \|x_i\|$. If $A_i \in B(E_i)$ are operators such that $\sup_i \|A_i\| <\infty$, its ultraproduct $\prod_\cU A_i$ is the operator sending $(x_i)_\cU$ to $(A_i x_i)_\cU$. This defines an isometric map $\prod_\cU B(E_i) \to B(\prod_\cU E_i)$, which is not surjective in general.

If $C \geq 1$, two Banach spaces $Y,Y'$ are said to be $C$-isomorphic if there is a continous and bijective linear map $u \colon Y \to Y'$ such that $\|u\| \|u^{-1}\| \leq C$. The Banach-Mazur distance (or isomorphism constant) between $Y$ and $Y'$ is the infimum of the constants $C$ such that $Y$ and $Y'$ are $C$-isomorphic. We warn the reader that the Banach-Mazur distance is sub-multiplicative. So to get a distance satisfying the usual triangle inequality, one should (but we will \emph{not}) take the log of $C$. The Banach-Mazur distance is particularilly relevant for spaces of the same finite dimension, as it typically infinite between infinite dimensional Banach spaces.

A Banach space $X$ is finitely representable in a class $\cE$ of
Banach spaces if for every finite dimensional subspace $Y$ of $X$ and
every $\varepsilon>0$, $Y$ is at Banach-Mazur distance less than
$1+\varepsilon$ from a subspace of a space in $\cE$ (that is, there is
a space $X' \in \cE$ and a linear map $u \colon Y \to E$ such that
$(1-\varepsilon) \|x\| \leq \| u(x)\| \leq \|x\|$ for all $x \in
Y$). For example, $L_p([0,1])$ is finitely presentable in $\ell_p$,
and $\ell_p$ is fintely representable in $L_p([0,1])$. Recall
\cite{dacunhacastellekrivine} or \cite{Heinrich} that $X$ is finitely
representable in $\cE$ if and only if $X$ is isometrically isomorphic
to a subspace of an ultraproduct of spaces in $\cE$. For the
convenience of the reader not familiar with the local theory of Banach
spaces, we reproduce here the standard proof of this equivalence. We
will use similar arguments later in the paper.

\begin{proof}  Assume that $X$ is finitely representable in $\mathcal E$. Denote by $I$ the set of all pairs $(\varepsilon,Y)$ for $\varepsilon >0$ and $Y$ a finite dimensional subspace of $X$. Declare that $(\varepsilon,Y)$ is larger than $(\varepsilon',Y')$ if $\varepsilon < \varepsilon'$ and $Y' \subset Y$. This is an order relation which makes $I$ into a directed set. Let $\cU$ be a cofinal ultrafilter on $I$. By assumption, for every $i \in I$, there is a Banach space $E_i \in \cE$ and a linear map $u_i \colon Y \to E_i$ such that $(1-\varepsilon) \|x\| \leq \| u_i(x)\| \leq \|x\|$ for all $x \in Y$. Extend $u_i$ to a nonlinear map $X \to E_i$ by setting $u_i(x) = 0$ if $x \notin E$. This allows to define a map $u \colon E \to \prod_\cU E_i$ by defining $u(x)$ as the class of $(u_i(x))_i$. The cofinality of $\cU$ implies that $u$ is linear and isometric.

  For the converse, assume that there is a set $I$ with ultrafilter $\cU$, a family $(X_i)_i \in I \in \cE^I$ and an isometric embedding $u$ of $X$ into $\prod_\cU X_i$. Let $Y$ be a finite dimensional subspace of $X$, and $y_1,\dots,y_n$ be a basis for $Y$. For every $k \leq n$, pick $(y_{k,i})_i \in \prod_i X_i$ a representative of $u(y_k)$. Define, for every $i$, a linear map $u_i \colon Y \to X_i$ by extending by linearity the map $y_k \mapsto y_{k,i}$. Then by linearity of $u$, for every $y \in Y$, $(u_i(y))_i$ is a representative of $u(y)$, and in particular $\lim_\cU \|u_i(y)\|=y$. By compactness the convergence in uniform in the unit ball of $y$, and in particular for every $\varepsilon>0$, there is $i$ such that
  \[ (1-\varepsilon) \|y\| \leq \|u_i(y)\| \leq (1+\varepsilon) \|y\|.\]
This proves that $X$ is finitely representable in $\mathcal E$.
\end{proof}

If $\cE$ is a class of Banach spaces, we denote by $\cE^{N,\varepsilon}$ the class of Banach spaces $E$ such that all $N$-dimensional subspaces of $E$ are at Banach-Mazur distance less than $1+\varepsilon$ from a subspace of a space in $\cE$.

We shall use the following finitary version of the well-known fact that the dual of a subspace is isometric to a quotient of the dual.
\begin{lem}\label{lem:duality_finite_representability} For every $N\in \N, \varepsilon>0$, there exists $N' \in \N,\varepsilon'>0$ such that the following holds. If $\cE$ is a class of Banach spaces and $X \in \cE^{N',\varepsilon'}$, then every subspace of $X^*$ of dimension $\leq N$ is $(1+\varepsilon)$-isometric to a subquotient (=subspace of the quotient) of the  dual of a space in $\cE$.

In particular if $\cE$ is stable under subspaces and duals, then for every Banach space $X \in \cE^{N',\varepsilon'}$, $X^* \in \cE^{N,\varepsilon}$.
\end{lem}

For the proof, we shall need the following classical lemma, to which we provide a proof for the reader's convenience.
\begin{lem}
  For every $N\in \N, \varepsilon>0$, there exists $N' \in \N>0$ such that, for every Banach space $X$ and every subspace $Y \subset X^*$ of dimension $ \leq N$, there is a subspace $Z \subset X$ of dimension $\leq N'$ such that $Y^*$ is $(1+\varepsilon)$-isomorphic to a subspace of $Z^*$.
\end{lem}
\begin{proof} By compactness, there is $N'$ such that for every Banach space $Y$ of dimension $N$, its unit sphere $S_Y$ contains an $\varepsilon$-net $F$ or cardinality $\leq N'$: $S_Y \subset \cup_{x \in F} B(x,\varepsilon)$.

  Let $X$, $Y$ be as in the lemma, and $F \subset S_Y$ an
  $\varepsilon$-net as above. For every $y \in F$, let $x_y \in X$ be
  a norm one element on which $y$ almost attains its norm: $|\langle
  y,x_y\rangle | \geq (1-\varepsilon)$. Let $y_0$ in the unit sphere
  of $Y$. There is $y \in F$ such that $\|y_0-y\| \leq \varepsilon$,
  and therefore
  \[|\langle y_0,x_y\rangle| \geq  |\langle y,x_y\rangle|-| \langle y_0-y,x_y\rangle| \geq 1-2\varepsilon.\]
  In particular, if we define $Z \subset X$ as the linear span of the
  $x_y$'s, we have that the norm on $y_0$ in $Z^*$ is at least
  $(1-2\varepsilon)$. In other words, the formal inclusion
  (restriction of linear forms to $Z$) $u \colon Y \to Z^*$ satisfies
  \[(1-2\varepsilon) \|y\| \leq \|u(y)\| \leq \|y\|\]
  for every $y \in Y$, and $Y$ is $(1-2\varepsilon)^{-1}$-isomorphic to a subspace of $Z^*$. This is the Lemma, up to a change of $\varepsilon$.\end{proof}

\begin{proof}[Proof of Lemma \ref{lem:duality_finite_representability}]
  Fix $N \in \N,\varepsilon>0$, and let $N'$ be given by the preceding lemma for $\varepsilon/2$. Pick $\varepsilon'>0$ such that $(1+\varepsilon/2)(1+\varepsilon') \leq (1+\varepsilon)$.

  Let $X \in \cE^{N',\varepsilon'}$, and $Y$ be a subspace of dimension $\leq N$ of $Y^*$. The preceding lemma provides a subspace $Z \subset X$ of dimension $\leq N'$ such that $Y$ is $(1+\varepsilon/2)$-isomorphic to a subspace of $Z^*$. But by the definition of $\cE^{N',\varepsilon'}$, $Z$ is $(1+\varepsilon')$-isomorphic to a subspace $Z'$ of a space $X'$ in $\mathcal E$. This implies $Z^*$ is $(1+\varepsilon')$-isomorphic to $Z'^*$, which is a quotient of $X'^*$. Putting everything together, $Y$ is $(1+\varepsilon/2)(1+\varepsilon')$-isomorphic to a subspace of a quotient of $X'^*$. This proves the first part of the lemma.
  
  The second  part is immediate because a subquotient of the  dual of a space in $\cE$ belongs to $\cE$ if $\cE$ is stable under subspaces and duals.
\end{proof}

\subsection{Superreflexivity}
\label{subsection:superreflexivity} 
A Banach space $X$ (respectively a class $\mathcal E$ of Banach spaces) is said to be superreflexive if every Banach space finitely representable in $X$ (respectively $\mathcal E$) is reflexive. By a celebrated theorem of Enflo \cite{enflo} (see also \cite{pisier}), $X$ is superreflexive if and only if it carries an equivalent uniformly convex norm. Recall that a Banach space is said to be uniformly convex if its modulus of uniform convexity is strictly positive, and that the modulus of uniform convexity of a Banach space $X$ is the function
\[ t \in (0,1) \mapsto \inf\left\{ 1 - \| \frac{x+y}{2}\|  \mid x,y \in X,\|x\| \leq 1, \|y\|\leq 1, \|x-y\|\geq t\right\}.\] By an $\ell_2$-direct sum argument, one obtains that a class of Banach spaces is superreflexive if and only if there is a constant $C$ and of a function $\delta\colon (0,1) \to (0,1)$ such that every space in $\cE$ is $C$-isomorphic to a uniformly convex space with modulus of uniform convexity $\geq \delta$.
\section{A local characterization of (T) and its variants}

\subsection{Definition of Kazhdan projection}
Let $\cF$ be a class of representations of $G$ such that
\[ \sup_{(\pi,E) \in \cF}\|\pi(g)\|_{B(E)}\] is bounded on compact subsets of $G$.

We define $\cC_{\cF}(G)$ as the completion of $C_c(G)$ for the norm
\[ \|f\|_{\cF} = \sup\{ \|\pi(f)\|_{B(E)}, (\pi,E) \in \cF\}.\]
It is a Banach algebra for convolution. By construction, the map $f \in C_c(G) \mapsto \pi(f) \in B(E)$ extends uniquely to a bounded map on $\cC_{\cF}(G)$, that we still denote by $\pi$. For example, if $\cF$ is the class of unitary representations 
of $G$ then $\cC_{\cF}(G)$ is the full $C^*$-algebra of $G$.

\begin{defn}\label{defn:Kazhdan_proj} A Kazhdan projection in $\cC_{\cF}(G)$ is an idempotent $p$ belonging to the closure of $\{f \in C_c(G), \int f =1\}$ such that $\pi(p)$ is a projection on $E^\pi$ for every $(\pi,E) \in \cF$.

A Kazhdan projection is called central if it belongs to the center of $\cC_{\cF}(G)$.
\end{defn}
\begin{rem}\label{rem:comment_K_proj} The assumption that $p$ belongs to the closure $\{f \in C_c(G), \int f =1\}$ is just here to make a nontrivial definition in the case $E^\pi =\{0\}$ for every $(\pi,E)\in\cF$ (otherwise we could just take $p=0 \in C_c(G)$). This assumption is superfluous otherwise~: if $p \in \cC_{\cF}(G)$ is such that $\pi(p)$ is a projection on $E^\pi$ for every $(\pi,E) \in \cF$, and if $E^\pi \neq \{0\}$ for at least one $(\pi,E)$, then $p$ belongs to the closure of $\{f \in C_c(G), \int f =1\}$. Indeed, if $x\in E^\pi \setminus \{0\}$ and $f_n \in C_c(G)$ converges to $p$, we have
\[ x= \pi(p) x = \lim_n \pi(f_n) x = \lim_n (\int f_n) x,\]
so that $\lim_n \int f_n=1$, and $p = \lim_n \frac{f_n}{\int f_n}$ belongs to the closure of  $\{f \in C_c(G), \int f =1\}$.

We insist that for us Kazhdan projections are not necessarily central. One reason is that, as indicated in Proposition \ref{prop:central}, the question whether a Kazhdan projection is central is essentially disjoint from the question whether there exists a Kazhdan projection. Another reason is that there is a natural setting where non central Kazhdan projections occur naturally. See Remark \ref{rem:noncentralKP}.
\end{rem}
It is useful to realize that being a Kazhdan projection in $\cC_{\cF}(G)$ only depends on the norm $\|\cdot\|_{\cF}$, and not on the specific $\cF$.
\begin{lem}\label{lem:intrinsic_char_K_proj} An element $p \in \cC_{\cF}(G)$ is a Kazhdan projection if and only if it belongs to the closure of $\{f \in C_c(G), \int f =1\}$ and satisfies $f \ast p = (\int f) p$ for all $f \in C_c(G)$. It is a central Kazhdan projection if and only it moreover satisfies $p\ast f = (\int f) p$ for all $f \in C_c(G)$.
\end{lem}
\begin{proof}
For every $(\pi,E)$ in $\cF$, a vector $x \in E$ belongs to $E^\pi$ if and only if $\pi(f)x= (\int f) x$ for all $f \in C_c(G)$. Therefore, if $p$ is a Kazhdan projection, we have that $\pi(f \ast p) = \pi(f) \pi(p) = (\int f) \pi(p)$ for every $(\pi,E) \in \cF$, and hence $f \ast p = (\int f )p$. Conversely, if $p$ belongs to the closure of $\{f \in C_c(G), \int f =1\}$ and satisfies $f \ast p = (\int f) p$, then $p$ is an idempotent, and for every $(\pi,E)\in\cF$ we have that $\pi(p)$ acts as the identity on $E^\pi$ (because every $f \in C_c(G)$ with $\int f=1$ does), and its image is made of invariant vectors (because $\pi(f) \pi(p) = (\int f) \pi(p)$ for every $f \in C_c(G)$). Thus $p$ is a Kazhdan projection. The statement about central projections is immediate.
\end{proof}
\subsection{To be central or not to be}\label{sec:hamlet} We now discuss when a Kazhdan projection is central.

\begin{prop}\label{prop:central} Assume that $\cC_{\cF}(G)$ contains a Kazhdan projection $p$. The following are equivalent.
\begin{enumerate}
\item\label{item:central} $p$ is a central Kazhdan projection.
\item\label{item:invariant_complement} for every representation $(\pi,E)$ in $\cF$, $E^\pi$ has a $\pi(G)$-invariant complement subspace in $E$.
\item\label{item:only_proj} $p$ is the only Kazhdan projection in $\cC_{\cF}(G)$.
\item\label{item:exists_rightKp} There is an element $p'$ in the closure of $\{f \in C_c(G), \int f=1\}$ such that $p' \ast f= (\int f) p'$ for every $f \in C_c(G)$.
\item\label{item:exists_dualKP} $\cC_{\cF^*}(G)$ contains a Kazhdan projection.
\end{enumerate}
\end{prop}
\begin{proof} (\ref{item:central})$\implies$(\ref{item:invariant_complement}). If $p$ is central, then $\mathrm{ker}(\pi(p))$ is a complement subspace of $E^\pi$, and it is invariant by $\pi(f)$ for every $f \in C_c(G)$ because $ \pi(f) \pi(p) = \pi(p)\pi(f)$. By strong continuity of $\pi$, $\mathrm{ker}(\pi(p))$ is therefore invariant by $\pi(g)$ for every $g \in G$.

(\ref{item:invariant_complement})$\implies$(\ref{item:only_proj}). Assume that for every $(\pi,E) \in \cF$, $E^\pi$ has a $\pi(G)$-invariant complement $F_\pi$. Then for every $f \in C_c(G)$ (and therefore for every $f \in \cC_{\cF}(G)$), $\pi(f) F_\pi \subset F_\pi$. In particular if $p'$ is a Kazhdan projection in $\cC_{\cF}(G)$, then $\pi(p') F_\pi \subset F_\pi \cap E^\pi = \{0\}$, so that $\pi(p')$ is the projection on $E^\pi$ parallel to $F_\pi$. In particular $\pi(p) = \pi(p')$, and therefore $p=p'$.

(\ref{item:only_proj})$\implies$(\ref{item:central}) If $p$ is a Kazhdan projection, then for every $f \in C_c(G)$ with $\int f=1$, $p \ast f$ is another Kazhdan projection, so that by (\ref{item:only_proj}) $p \ast f=p=f\ast p$. By linearity we deduce that $p \ast f=f\ast p$ for every $f \in C_c(G)$, and that $p$ is central.

(\ref{item:central})$\implies$(\ref{item:exists_rightKp}) is obvious (take $p'=p$). For the converse, suppose (\ref{item:exists_rightKp}). Then $pp'$ belongs to the closure of $\{f, \int f=1\}$ and satisfies $f_1 p p' f_2 = (\int f_1)  (\int f_2) pp'$, so that by Lemma \ref{lem:intrinsic_char_K_proj} $pp'$ is a central Kazhdan projection. By the already proven implication (\ref{item:central})$\implies$(\ref{item:only_proj}) for $pp'$, we deduce that $pp'$ is the only Kazhdan projection, so that $p=pp'$, and $p$ is central (\ref{item:central}).

(\ref{item:exists_rightKp})$\iff$(\ref{item:exists_dualKP}). For $f \in C_c(G)$, define $\widetilde f \in C_c(G)$ by the property that $\widetilde f dg$ is the image of the measure $fdg$ by the map $g \mapsto g^{-1}$. For every representation $(\pi,E)$ and $f \in C_c(G)$, the equality $\|\pi(f)\| = \|\tr \pi(\widetilde f)\|$ holds because $\tr \pi(\widetilde f)$ is the restriction to the weak-* dense subspace $\tr E\subset E^*$ of the weak-* continuous operator $\pi(f)^* \in B(E^*)$. The map $f \mapsto \widetilde f$ therefore extends to a surjective isometry $\cC_{\cF}(G) \to \cC_{\cF^*}(G)$. It preserves $\{f\in C_c(G), \int f=1\}$ and satisfies $\widetilde{f_1 \ast f_2} = \widetilde{f_2} \ast \widetilde {f_1}$. It is immediate by Lemma \ref{lem:intrinsic_char_K_proj} that an element $p' \in \cC_\cF(G)$ satisfies (\ref{item:exists_rightKp}) if and only if $\widetilde{p'}$ is a Kazhdan projection in $\cC_{\cF^*}(G)$.
\end{proof}

We say that $\cF$ is \emph{weakly self-adjoint} if the the norms $\|f\|_\cF$ and $\|f\|_{\cF^*}$ are equivalent on $C_c(G)$. By the proof of (\ref{item:exists_rightKp})$\iff$(\ref{item:exists_dualKP}) in the preceding proposition, $\cF$ is weakly self-adjoint if and only if the map $f \mapsto \widetilde f$ extends to a bounded map on $\cC_{\cF}(G)$ that we still denote by $\widetilde{\cdot}$. An element $a \in \cC_\cF(G)$ is then called self-adjoint if $\widetilde a=a$. For example, $\cF$ is weakly self-adjoint if $\cF^* \subset \cF$.

\begin{cor}\label{cor:KP_automatically_central} If $\cF$ is weakly self-adjoint, a Kazhdan projection in $\cC_{\cF}(G)$ is automatically central and self-adjoint.

If moreover the map $f \mapsto \overline{f}$ extends to a continuous map on $\cC_{\cF}(G)$, then a Kazhdan projection is automatically central, self-adjoint and real.
\end{cor}
\begin{proof}

Let $p \in \cC_\cF(G)$ a Kazhdan projection. Then $\widetilde p$ satisfies (\ref{item:exists_rightKp}) in the previous proposition, so $p$ is central (\ref{item:central}). Then $\widetilde p$ is another central Kazhdan projection, so by the previous proposition again $p=\widetilde p$ and $p$ is self-adjoint.

If moreover $f \mapsto \overline{f}$ extends by continuity to $\cC_\cF(G)$ then $\overline p$ is also a Kazhdan projection so by the previous proposition again $p=\overline p$ and $p$ is real.
\end{proof}

\subsection{Examples of Kazhdan projections}\label{subsection:examples}
The first example is for unitary representations~: it is classical that $G$ has property (T) if and only if the full $C^*$-algebra of $G$ has a Kazhdan projection. More generally, if $\cF$ is a class of isometric representations on a superreflexive set of Banach spaces, then $\cC_{\cF}(G)$ has a (necessarily central by Proposition \ref{prop:central} (\ref{item:invariant_complement})) Kazhdan projection if and only if there exists $\varepsilon>0$ such that $\max_{g \in S}\|\pi(g) x-x\| \geq \varepsilon \|x\|_{E/E^\pi}$ for every $(\pi,E) \in \cF$ and $x \in E$, \emph{i.e.} if and only if $E/E^\pi$ does not have almost invariant vectors uniformly in $(\pi,E) \in \cF$. See \cite[Theorem 1.2]{drutunowak} (the argument in the case of a discrete group was previously recorded in \cite[Proposition 5.1]{delaatdelasalle2}).

Let $(\pi,E)$ be a Banach space representation of $G$. An argument of Guichardet, originally used for unitary representations but valid for arbitrary Banach space representations, shows that if $H^1(G;\pi)=0$ then $E/E^\pi$ does not have almost invariant vectors\footnote{The argument goes as follows~: if $H^1(G;\pi)=0$, the map $x \in E/E^\pi \mapsto (\pi(g) x-x)_{g \in G}$ with values in the Banach space $Z^1(G;\pi)$ with the norm $\sup_{g \in S}\|b(g)\|$ is continuous and bijective, and hence invertible by the open mapping theorem. If $C$ is the norm of the inverse we have $\delta_S^\pi(x) \geq C^{-1} \|x\|_{E/E^\pi}$ for every $x \in E$.}. We shall see (this was observed by Masato Mimura \cite{mimura}) by some ultraproduct argument that in several case, this holds uniformly in $(\pi,E)$ (Lemma \ref{lem:FE_implies_uniform_spectral_gap}). In particular, in \S \ref{subsection:fixed_point_KP} we will deduce the following result. 
\begin{prop}\label{prop:FE_implies_KazhdanProj} Let $\cE$ be a class of superreflexive Banach spaces, and denote by $\cF$ all isometric representations of $G$ on a space in $\cE$. 
If $G$ has (F$_{\cE}$) then $\cC_{\cF}(G)$ has a Kazhdan projection in the following situations:
\begin{itemize}
\item $\cE$ is stable by finite representability,
\item or $\cE$ is the class of $L_p$-spaces for some $1<p<\infty$,
\item or $G$ is discrete and $\cE$ is stable under ultraproducts.
\end{itemize}
\end{prop}

Finally, examples where Kazhdan projections occur are in the definitions of Banach strong property (T) and its variant robust property (T). Let us recall the definitions.

If $\cE$ is a class of Banach spaces and $m \colon G \to (0,\infty]$ is a function, we denote $\cF({\cE},m)$ all equivalence classes of representations $(\pi,E)$ such that $E\in \cE$ and $\|\pi(g)\|\leq e^{m(g)}$ for all $g$.

Vincent Lafforgue's strong property (T) \cite{lafforguestrongt} was defined in terms of Kazhdan projection. 
\begin{defn}(Lafforgue) If $\cE$ is a class of Banach spaces, one says that $G$ has strong property (T) with respect to $\cE$ is there there is $s>0$ such that for all $C>0$, $\cC_{\cF(\cE,s\ell +C)}(G)$ has a Kazhdan projection.
\end{defn}
Lafforgue originally only considered the case when $\cE$ is stable by duality, subspaces and complex conjugation and wanted the Kazhdan projection to be self-adjoint and real, but in that case the two definitions coincide by Corollary \ref{cor:KP_automatically_central}.

Oppenheim \cite{oppenheim} defined robust property (T) with respect to a class of Banach spaces $\cE$, as an intermediate property between property (T) and strong property (T) with respect to $\cE$. An almost
\footnote{In his original definition, Oppenheim requests, as in \cite{delasalle1}, that the Kazhdan projection belong to the closure of the symmetric functions on $G$. In view of Corollary \ref{cor:KP_automatically_central}, this is automatic if $\cF(\cE,c\ell)$ is weakly self-adjoint. This does not seem any more so relevant otherwise, so we prefer to drop this condition.} equivalent form of his definition is the following.
\begin{defn}(Oppenheim)\label{defn:robustT} $G$ has robust property (T) with respect to $\cE$ if there exists $s>0$ such that $\cC_{\cF(\cE,c\ell)}(G)$ has a Kazhdan projection.
\end{defn}
This is equivalent to the Kazhdan projection being central when $\cF(\cE,c\ell)$ is weakly self-adjoint, for example if $\cE$ is stable by duality and made of reflexive spaces.

\subsection{The main theorem} If $(\pi,E)$ is a representation of $G$, for every $x \in X$ we denote 
\begin{equation}\label{eq:def_deltaS} \delta_S^\pi(x) =\max_{g \in S} \|\pi(g) x - x\|.\end{equation}
Observe that by the triangle inequality, if $g \in S^N$ then 
\begin{equation}\label{ineq_delta} \| \pi(g) x - x \| \leq N \sup_{g \in S} \|\pi(g)\|^{N-1} \delta_S^\pi(x).\end{equation}
We will also consider the quantity $\delta_S^\sigma(x) =\max_{g \in Q} \|\sigma(g) x - x\|$ when $\sigma$ is an affine action of $G$ on $E$ with linear part $\pi$, and in that case \eqref{ineq_delta} still holds in the form  $\| \sigma(g) x - x \| \leq N \sup_{g \in S} \|\pi(g)\|^{N-1} \delta_S^\sigma(x)$.

The following is Theorem \ref{thm=local_characterization_of_Kazhdan_constant} generalized to arbitrary locally compact groups.
\begin{thm}\label{thm=local_characterization_of_Kazhdan_constant_nondiscrete} The following are equivalent~:

\begin{enumerate}
\item\label{item=kazhdan} $\cC_{\cF}(G)$ contains a Kazhdan projection.
\item\label{item=local_shrinking} There is a compactly supported measure $m$ with $\int 1 dm=1$ such that $\delta_S^\pi(\pi(m) x) \leq \frac 1 2 \delta_S^\pi( x)$ for all $(\pi,E)\in \cF$ and $x \in E$.
\end{enumerate}

If these properties hold and if $(\sigma,E)$ is an affine action of $G$ whose linear part belongs to $\cF$, then $\sigma$ has a fixed point if and only if $\delta_S^\sigma(\sigma(m) x) \leq \frac 1 2\delta_S^\sigma(x)$ for all $x \in E$.
\end{thm}

We first record an easy fact on the displacement \eqref{eq:def_deltaS}, that we will often use.
\begin{lem}\label{lem=trivialite} Let $m$ be a compactly supported measure on $G$ with $\int 1 dm =1$. There is a constant $C_m$ such that 
\[ \| \sigma(m) x - x \|_E \leq C_m \delta_S^\sigma(x)\]
for every affine action $\sigma$ of $G$ with linear part $(\pi,E)\in \cF$ and every $x \in E$.
\end{lem}
\begin{proof} Let $N$ be such that the support of $m$ is contained in $S^N$, and $\|m\|_{TV}$ the total variation norm of $m$. Then \eqref{ineq_delta} implies
\[ \| \sigma(m) x - x \|_E \leq \int \|\sigma(g) x - x \|_E |dm| \leq \|m\|_{TV} N \sup_{g \in S} \|\pi(g)\|^{N-1} \delta_S^\pi(x).\]
This proves the lemma because $\sup_{g \in S} \|\pi(g)\|^{N-1}$ is bounded independantly of $(\pi,E)\in \cF$.
\end{proof}

For the proof of the direction (\ref{item=kazhdan})$\implies$(\ref{item=local_shrinking}) we will need a uniform version of the fact that the existence of a Kazhdan projection for $(\pi,E)$ implies that $(\pi,E)$ does not almost have invariant vectors.
\begin{lem}\label{lem=kazhdan_implies_noaivector}
If $\cC_{\cF}(G)$ contains a Kazhdan projection, then there exists $c>0$ such that
\[ \delta_S^\pi(x) \geq c \|x\|_{E/E^\pi} \ \forall (\pi,E) \in \cF\textrm{ and }x \in E.\]
\end{lem}
\begin{proof} Let $p \in \cC_{\cF}(G)$ be a Kazhdan projection and $f \in C_c(G)$ such that $\int f=1$ and $\|f - p \|_{\cF}\leq 1/2$. Then for every $x \in E$, 
\[ \|x\|_{E/E^\pi} \leq \| x - \pi(p) x\|_E \leq \|x-\pi(f) x\|_E + \|\pi(f-p)x\|_E \leq \|x-\pi(f)x\|_E + \frac 1 2 \|x\|_E.\]
By replacing $x$ by $x+y$ for $y \in E^\pi$ in the preceding equation, the term $\|x - \pi(f) x\|_E$ is unchanged because $\int f=1$ and we get
\[ \|x\|_{E/E^\pi} \leq  \|x-\pi(f)x\| + \frac 1 2 \|x+y\|_E.\]
Taking the infimum over $y \in E^\pi$ we obtain $\|x\|_{E/E^\pi}\leq 2 \|x-\pi(f)\|_E$. We conclude by Lemma \ref{lem=trivialite} for the measure $f dg$.
\end{proof}

\begin{proof}[Proof of Theorem \ref{thm=local_characterization_of_Kazhdan_constant}] (\ref{item=kazhdan})$\implies$(\ref{item=local_shrinking}). Let $\varepsilon>0$ to be determined later. Let $p \in \cC_{\cF}(G)$ be a Kazhdan projection, and $f \in C_c(G)$ such that $\int f=1$ and $\|p - f\|_{\cF} <\varepsilon$. We prove (\ref{item=local_shrinking}) for the measure $m =fdg$ if $\varepsilon$ is small enough. Define $C$ by 
\[C=\sup_{(\pi,E) \in \cF} \max_{g \in S} \|\pi(g)\|_{B(E)}.\]
Then for $(\pi,E) \in \cF$
\[ \delta_S^\pi(\pi(f) x) = \delta_S^\pi(\pi(f-p) x) \leq (1+C)\|\pi(f-p) x\|_E \leq (1+C)\varepsilon \|x\|_E.\]
By applying this inequality to $x+y$ for $y \in E^\pi$ and taking the infimum over all $y$ we get
\[ \delta_S^\pi(\pi(f)x) \leq (1+C) \varepsilon \|x\|_{E/E^\pi},\]
which by lemma \ref{lem=kazhdan_implies_noaivector} is less than $\frac{(1+C) \varepsilon}{c}\delta_S^\pi(x)$. This is less than $\frac 1 2$ for $\varepsilon <\frac{c}{2+2C}$.

(\ref{item=local_shrinking})$\implies$(\ref{item=kazhdan}). Let $(\pi,E) \in \cF$. By iterating the inequality in (\ref{item=local_shrinking}), we get that $\delta_S^\pi( \pi(m)^n x) \leq 2^{-n} \delta_S^\pi(x)$. If $C_m$ is the constant given by lemma \ref{lem=trivialite} for $m$, we obtain $\|\pi(m)^{n+1} x-\pi(m)^n x\|\leq 2^{-n} C_m \delta_S^\pi(x)$.
By bounding 
\[ \delta_S^\pi(x) \leq  (1+\sup_{(\pi,E) \in \cF} \max_{g \in S} \|\pi(g)\|) \|x\|_E = (1+C)\|x\|_E,\]
we get $\|\pi(m)^{n+1} - \pi(m)^n\| \leq C_m(1+C) 2^{-n}$. This implies that $\pi(m)^n$ is a Cauchy sequence in $B(E)$, and hence has a limit $P_\pi \in B(E)$ and $\| \pi(m)^n - P_\pi\|_{B(E)} \leq 2^{-n} C'$, for $C'=2 C_m (1+C)$. Then $\delta_S^\pi(P_\pi x) = \lim_n \delta_S^\pi(\pi(m)^nx)=0$ and $P_\pi x \in E^\pi$ because $S$ generates $G$. Since for $x \in E^\pi$, $P_\pi(x) =\lim_n \pi(m)^n x = \lim_n x=x$, we get that $P_\pi$ is a projection on the invariant vectors. We are almost done, except that $m^{\ast n} = m\ast m\ast \dots m$ might not be absolutely continuous with respect to the Haar measure. This can be fixed by choosing a function $f_0 \in C_c(G)$ with $\int f_0=1$, and observing that $f_0\ast m^{\ast n}$ belongs to $C_c(G)$ and is Cauchy in $\cC_{\cF}(G)$. Its limit $p$ satisfies $\pi(p)=\pi(f_0) P_\pi =P_\pi$, \emph{i.e.} $p$ is a Kazhdan projection.

Now assume that (\ref{item=kazhdan}) and (\ref{item=local_shrinking}) hold, and let $(\sigma,E)$ be an affine action of $G$, the linear part of which belongs to $\cF$. If $\sigma$ has a fixed point, then $\sigma$ is just a representation in $\cF$ in which the origin has been renamed, so that the inequality $\delta_S^\sigma(\sigma(m) x) \leq \frac 1 2 \delta_S^\sigma( x)$ is immediate from (\ref{item=local_shrinking}). Conversely, if $\delta_S^\sigma(\sigma(m) x) \leq \frac 1 2 \delta_S^\sigma( x)$ for all $x \in E$ then the proof of (\ref{item=local_shrinking})$\implies$(\ref{item=kazhdan}) shows that $\sigma(m^{\ast n})x$ is a Cauchy sequence, and hence converges to a point $y$ satisfying $\delta_S^\sigma(y) = \lim_n \delta_S^\sigma(\sigma(m^{\ast n})x) = 0$, \emph{i.e.} to a fixed point. 
\end{proof}

Finally, we record the following corollary of the proof of Theorem \ref{thm=local_characterization_of_Kazhdan_constant}, which is essential for applications to dynamics \cite{brownFisherHurtado} (see also the analogous discussion below for positive Kazhdan constants). 
\begin{cor}\label{cor:Convergence_speed_to_KP} If $\cC_{\cF}(G)$ contains a Kazhdan projection, then there are $C,s>0$ and a Kazhdan projection $p\in\cC_{\cF}(G)$ such that $p$ is at distance $\leq C e^{-sn}$ from the continuous functions supported in $S^n$.
\end{cor}
\begin{proof} In the proof of (\ref{item=local_shrinking})$\implies$(\ref{item=kazhdan}) in Theorem \ref{thm=local_characterization_of_Kazhdan_constant}, we constructed $p$ as the limit of $f_0 \ast m^{\ast n}$, where $f_0 \in C_c(G)$, $m$ is a compactly supported measure and $\|p - f_0 \ast m^{\ast n}\| \leq C' 2^{-n}$. This proves the corollary.
  \end{proof}

\subsection{Positive Kazhdan projections}

A natural variant of Definition \ref{defn:Kazhdan_proj} is to require additionally that $p$ belongs to the closure of the \emph{nonnegative} functions $\{ f \in C_c(G), f \geq 0, \int f =1\}$. This variant is particularily relevant to the applications to dynamics, see \cite{brownFisherHurtado}. To our knowledge, in all examples where a Kazhdan projection is known to exist (for examples strong property (T), or Corollaries \ref{cor:T_implies_robustT} -\ref{cor:Delorme-Guichardet_Banach}), it belongs to the closure of the nonnegative functions. However, we do not know if this is the case in general.

We can note that the proof of Theorem \ref{thm=local_characterization_of_Kazhdan_constant} shows in full generality that $\cC_{\cF}(G)$ contains such a ``positive'' Kazhdan projection if and only if there is a \emph{positive} compactly supported probability measure $m$ such that $\delta_S^\pi(\pi(m) x) \leq \frac 1 2 \delta_S^\pi( x)$ for all $(\pi,E)\in \cF$ and $x \in E$. Moreover, in that case there are $C,s>0$ and a Kazhdan projection $p\in\cC_{\cF}(G)$ such that $p$ is at distance $\leq C e^{-sn}$ from the continuous nonnegative functions supported in $S^n$.

Also, we can give a positive answer to the previous question for central Kazhdan projections and isometric representations.
\begin{prop} Assume that $\mathcal F$ is made of isometric representations and is stable by complex conjugation. If $\cC_{\cF}(G)$ contains a central Kazhdan projection, then this central Kazhdan projection belongs to the closure of the nonnegative functions.
\end{prop}
\begin{proof} Assume that $\cC_{\cF}(G)$ contains a central Kazhdan projection $p$. By Proposition \ref{prop:central}, $p$ is real and there is a real valued function $f \in C_c(G)$ with $\int f=1$ and $\|f-p\|\leq \frac 1 2$. Moreover, by replacing every $(\pi,E)\in \cF$ by $(\pi\left|_{\ker p}\right.,\ker p)$ (which is indeed a representation because the projection is central, see Proposition \ref{prop:central}) we can assume that $p=0$. Let $f=a f_+ - b f_-$ be a decomposition of $f$ with $f_+$ and $f_-$ nonnegative with integral $1$, and $a, b$ are nonnegative real numbers such that $a-b=1$. Then $1/2 \geq \|f\| \geq a \|f_+\| - b \|f_-\|$. This implies that $\|f_+\| \leq \frac{1/2+b}{a} <1$ (here we use that $\pi$ is isometric to ensure $\|\pi(f_-)\| \leq 1$). Then the sequence of $n$-th power convolutions of $f_+$ is a sequence of nonnegative functions of integral $1$ which also converge to $0=p$.
\end{proof}

\section{A topology on the space of representations}\label{sec:topology}

The purpose of this section is to define a natural topology on sets of (equivalence classes of) Banach space valued representations of locally compact groups, and to characterize this topology in terms of ultraproduct representations.

\subsection{Definition of the topology}\label{subsection:def_topolofy}

Let $G$ be a locally compact group and $\cR$ be a set of equivalence classes of Banach space representations of $G$. 

For every $(\pi,E) \in \cR$, every $x_1,\dots,x_n \in E$, every compact subset $Q \subset G$ and every $\varepsilon>0$ we define $W_{x_1,\dots,x_n,Q,\varepsilon}(\pi,E) \subset \cR$ as the set of all representations $(\pi',E') \in \cR$ such that there is $ x'_1,\dots,x'_n \in E'$ such that 
\begin{equation}\label{eq:def_WxQe}\sup_{f_1,\dots,f_n}| \| \sum_{k=1}^n \pi'(f_k) x'_k\| - \| \sum_{k=1}^n \pi(f_k) x_k \| |  < \varepsilon \end{equation} where the supremum is over all $f_1,\dots,f_n \in C_c(G)$ supported in $Q$ and with $\|f_k\|_{L_1} \leq 1$. 

The sets $W_{x_1,\dots,x_n,Q,\varepsilon}(\pi,E)$ form a basis for a topology on $\cR$. This topology is not Hausdorff (every subrepresentation of $(\pi,E)$ belongs to the closure of $\{(\pi,E)\}$). If $\cR$ contains the trivial representation on the $0$-dimensional Banach space, then it is compact-but-not-Hausdorff for the stupid reason that $\cR$ is the only neighbourhood of $0$.

\begin{rem}\label{rem:equivalence_WxQe}
Because we are dealing with strongly continuous representation, \eqref{eq:def_WxQe} implies 
\[\sup_{m_1,\dots,m_n}| \| \sum_{k=1}^n \pi'(m_k) x'_k\| - \| \sum_{k=1}^n \pi(m_k) x_k \| |  < \varepsilon\]
where the supremum is over all measures $m_1,\dots,m_n$ supported in the interior of $Q$ and with total variation $\leq 1$. Conversely, \eqref{eq:def_WxQe} follows from the preceding inequality where the supremum is over all measures $m_1,\dots,m_n$ with finite support contained in $Q$ and with total variation $\leq 1$.
\end{rem}

The restriction to unitary representations of this topology is not exactly the usual Fell topology \cite[Appendix F]{bekkadelaharpevalette}, since for example the trivial representation of $G$ on $\C^2$ does not belong to the closure of the trivial representation on $\C$, whereas it belongs to the Fell topology closure of it. The next lemma in particular shows that a unitary representation $\pi$ belongs to the closure of another unitary representation $\rho$ if and only if $\pi$ is weakly contained in the sense of Zimmer in $\rho$ \cite{zimmer} (see also \cite[Appendix F]{bekkadelaharpevalette}). 
\begin{lem} Assume that $\cR$ is a set of unitary representations of $G$. A representation $(\pi,\mathcal H)\in  \cR$ belongs to the closure of $A \subset \cR$ if and only if for every orthonormal family $\xi_1,\dots,\xi_n \in \mathcal H$, every compact subset $Q \subset G$ and every $\varepsilon>0$, there is a representation $(\rho,\mathcal K) \in A$, an orthonormal family $\eta_1,\dots,\eta_n \in \mathcal \mathcal K$ such that $\max_{i,j} \max_{g \in Q} |\langle \pi(g) \xi_i,\xi_j\rangle - \langle \pi(g) \eta_i,\eta_j\rangle| < \varepsilon$. 
\end{lem}
\begin{proof} Clear, because if $\pi$ is a unitary representation,
\[ \|\sum_k \pi(m_k) \xi_k\|^2 = \sum_{k,l} \iint \langle \pi(h^{-1} g) \xi_k,\xi_l\rangle dm_k(g) dm_l(h).\]
\end{proof}

\begin{rem}\label{rem:closure_implies_fin_rep+bound}
If $(\pi,E)$ belongs to the closure of $\cF \subset \cR$, then
\begin{itemize} 
\item $\|\pi(m)\| \leq \sup_{(\pi',E') \in \cF} \|\pi'(m)\|$ for every compactly supported measure on $G$. In particular $\cC_{\cF}(G)$ and   $\cC_{\overline{\cF}}(G)$ are isometric. By Lemma \ref{lem:intrinsic_char_K_proj}, this implies that if $\cC_{\cF}(G)$ has a Kazhdan projection, then so has $\cC_{\overline{\cF}}(G)$.
\item $E$ is finitely representable in $\{E',(\pi',E') \in \cF\}$. In particular if $G=\{1\}$ is the trivial group, we have just defined a (classical) topology on sets of Banach spaces which is characterized as follows~: $X$ belongs to the closure of a subset $\cE$ if and only if $X$ is finitely representable in $\cE$. Hence (Subsection \ref{subsection:ultraproducts}) this topology can be characterized in terms of ultraproducts. In the rest of this section we show such a characterization in the case of an arbitrary group $G$. This will be rather straightforward for discrete groups, and technically more involved in the general case.
\end{itemize}
\end{rem}

Before that we investigate strong neighbourhoods for this topology in the sense of Definition \ref{defn:strong_neigh}.

\subsection{Examples of strong neighbourhoods}\label{subsection:strongneig_ex}
Here we characterize, in several examples, the strong neighbourhoods of subsets of $\cR$. We need to recall some notation. If $\cE$ is a class of Banach spaces and $m \colon G \to (0,\infty]$ is a function, we denote $\cF({\cE},m)$ (respectively $\cF(\overline{\cE},m)$) the set of all equivalence classes of representations $(\pi,E)$ in $\cR$ such that $E$ is isometric to a space in $\cE$ (respectively $E$ is finitely representable in $\cE$) and $\|\pi(g)\|\leq e^{m(g)}$ for all $g$. Recall the definition of $\cE^{N,\varepsilon}$ in \S \ref{subsection:ultraproducts}.

\begin{prop}\label{prop:strong_neigh_Em} Let $\cE$ be a class of Banach spaces and $m \colon G\to (0,\infty]$ a function. Every strong neighbourhood of $\cF(\overline{\cE},m)$ contains $\cF(\cE^{N,\varepsilon},m+\varepsilon \ell)$ for some $N \in \N$ and $\varepsilon>0$.
\end{prop}
\begin{proof} Let $\cF'$ be a strong neighbourhood of $\cF$. Assume by contradiction that for every pair $\alpha = (N,\varepsilon)$ of an integer $N$ and a positive number $\varepsilon$ there is a representation $(\pi_\alpha,E_\alpha)$ which belongs to $\cF(\cE^{N,\varepsilon},m+\varepsilon \ell)$ by not to $\cF'$. Then every accumulation point of this net (for the order $(N,\varepsilon) \leq (N',\varepsilon')$ if $N \leq N'$ and $\varepsilon \geq \varepsilon'$) belongs to $\cF(\overline{\cE},m)$ (see Remark \ref{rem:closure_implies_fin_rep+bound}). This is a contradiction with (\ref{item:strongneig1}) in Lemma \ref{lem:characterization_strong_neigh}.
\end{proof}
A particular case of the preceding proposition is worth mentioning.
\begin{prop}\label{prop:strong_neigh_Hilb} Let $\mathcal H$ be the class of Hilbert spaces and $m \colon G\to (0,\infty]$ a function. Every strong neighbourhood of $\cF(\mathcal H,m)$ contains $\cF(\cE(\varepsilon),m+\varepsilon \ell)$ for some  $\varepsilon>0$, where $\cE(\varepsilon)$ is the class of Banach spaces such that
\[\frac 1 2  \left(\|x+y\|^2 + \|x-y\|^2 \right) \leq (1+\varepsilon) \left(\|x\|^2 + \|y\|^2 \right) \ \ \forall x,y \in E.\]
\end{prop}
This is indeed a particular case because, since the parallelogram inequality characterizes the Hilbert spaces, for every $N,\varepsilon$ there is $\varepsilon'>0$ such that $\mathcal H^{N,\varepsilon}$ contains $\cE(\varepsilon')$.

If $\cE$ is a class of Banach spaces and $m \colon G \to (0,\infty]$ is a function, we denote $\cG({\cE},m)$ the set of all equivalence classes of representations $(\pi,E) \in \mathcal R$ which are equivalent to a subrepresentation of $(\pi',E')$ where $E' \in \cE$ and $\|\pi'(g)\| \leq e^{m(g)}$ for all $g$ in $G$.
\begin{prop}\label{prop:strong_neigh_discrete} Let $\cE$ be a class of Banach spaces stable by ultraproducts and $m \colon G\to (0,\infty]$ a function. Assume that $G$ is discrete. Then $\cG(\cE,m)$ is closed and every strong neighbourhood of $\cG(\cE,m)$ contains $\cG(\cE,m+\varepsilon \ell)$ for some $\varepsilon>0$.
\end{prop}
\begin{prop}\label{prop:strong_neigh_Lp} Let $1 < p < \infty$. Then $\cG(L_p,0)$ is closed and every strong neighbourhood of $\cG(L_p,0)$ contains $\cup_{q \in [p-\varepsilon,p+\varepsilon]} \cG(L_q,\varepsilon \ell)$ for some $\varepsilon>0$.
\end{prop}
We postpone the proof of these propositions to the end of the section, because their proof requires the material in the rest of the section.

\subsection{Ultraproduct of Banach space representations}
Let $G$ be a locally compact group. We define ultraproducts in the category of continuous Banach space linear representations of $G$ in the same way as for unitary representations in \cite{CherixCowlingStraub}.

Let $(\pi_i)_{i \in I}$ be a family of representations of $G$ on Banach spaces $E_i$. Assume also that $\sup_i \|\pi_i(g)\|_{B(E_i)}$ is bounded on compact subsets of $G$. Let $\cU$ be an ultrafilter on $I$. Let $\prod_{\cU} E_i$ be the Banach space ultraproduct of $E_i$, and, for $g \in G$, $\pi(g)$ the ultraproduct of $\pi_i(g)$, which makes sense because $\sup_i \|\pi_i(g)\|_{B(E_i)} < \infty$ (\S \ref{subsection:ultraproducts}). For $f \in C_c(G)$ we define $\pi(f)$ as the ultraproduct of $\pi_i(f)$, which makes sense because $\sup_i \|\pi_i(g)\|_{B(E_i)}$ is bounded on the support of $f$.

Then $\pi$ is a group morphism from $G$ to the invertible operators on $\prod_\cU E_i$. However it is in general not strongly continuous, and to fix this we consider $E_\cU$ the subspace of $\prod_\cU E_i$ defined as the closure of the space spanned by 
\[\{ \pi(f) x, x \in \prod_\cU E_i, f \in C_c(G)\}.\] It is straightforward (see Lemma \ref{lem=characterization_XU} for a stronger statement) that the space $E_\cU$ is invariant by $\pi(g)$ for all $g \in G$, and the restriction of $\pi$ to $E_\cU$ is a strongly continuous representation of $G$. We define the ultraproduct of $(\pi_i,E_i)$ with respect to $\cU$ as $(\pi_\cU,E_\cU)$, where  $\pi_\cU$ is the restriction of $\pi$ to $E_\cU$.

We also have the following characterization of $E_\cU$ (see \S \ref{subsection:groupG} for the terminology).
\begin{lem}\label{lem=characterization_XU} Let $f_n \in C_c(G)$ be an approximate unit, and let $x  \in \prod_\cU E_i$ with representative $(x_i)_i \in \prod_i E_i$. Then the following are equivalent
\begin{enumerate}
\item\label{item:restriction1}$x \in E_\cU$.
\item\label{item:restriction2}$\lim_n \|\pi(f_n) x - x\|=0$.
\item\label{item:restriction3} For every $\varepsilon>0$, there is a neighbourhood $U$ of $e$ in $G$ such that \[\lim_\cU \sup_{g \in U} \|\pi_i(g) x_i - x_i\| \leq \varepsilon.\]
\end{enumerate}
\end{lem}
\begin{proof} The (\ref{item:restriction2})$\implies$(\ref{item:restriction1}) direction is obvious~: if $\lim_n \|\pi(f_n) x - x\|$ is equal to $0$, then $x = \lim_n \pi(f_n) x$ belongs to $E_\cU$. So is (\ref{item:restriction3})$\implies$(\ref{item:restriction2}) because if $Q_n$ denotes the support of $f_n$,
\[ \|\pi(f_n) x - x\| = \lim_\cU \|\pi_i(f_n) x_i - x_i\| \leq \lim_\cU \sup_{g \in Q_n} \|\pi_i(g) x_i-x_i\|.\]

Let us prove the implication (\ref{item:restriction1})$\implies$(\ref{item:restriction3}). Assume that $x \in E_\cU$. Fix $U_0$ a compact neighbourhood of the identity in $G$, and define $M_0 =  \sup_{g \in U_0} \sup_i \|\pi_i(g)\|$. Let $\varepsilon>0$, and take $y = \sum_k \pi(h_k) y^{(k)}$ with $h_k \in C_c(G)$ and $y^{(k)} \in \prod_{\cU} E_i$ such that $\| y - x\|<\varepsilon$. Let $(y^{(k)}_{i})_i \in \prod_i E_i$ be a representative of $y^{(k)}$, so that $(y_i = \sum_k \pi_i(h_k) y^{(k)}_{i})_i$ is a representative of $y$. Then for every $g \in G$,
\[ \|\pi_i(g) x_i- x_i\| \leq \sum_k \|\pi(\lambda_g h_k -h_k) y^{(k)}_{i}\| + (1+\|\pi(g)\|) \|y_i - x_i\|\]
where $\lambda_g h_k(g') = h_k(g^{-1} g')$. Let also $Q \subset G$ a compact subset which contains the support of $h_k$ and $\lambda_g h_k$ for all $g \in U_0$ and all $k$. Let $M= \sup_{g \in Q} \sup_i \|\pi_i(g)\|$, so that $\| \pi_i(\lambda_g h_k -h_k) \| \leq M \|\lambda_g h_k - h_k\|_{L_1(G)}$. For $g \in U_0$ the previous inequality becomes
\[ \|\pi_i(g) x_i- x_i\| \leq M \sum_k \|\lambda_g h_k -h_k\|_{L_1(G)} \|y^{(k)}_{i}\| + (1+M_0) \|y_i - x_i\|.\]
By continuity of the translations on $L_1(G)$, there exists $U \subset U_0$ a neighbourhood of the identity such that
\[ M \sum_k \|\lambda_g h_k -h_k\|_{L_1(G)} \|y^{(k)}\| \leq \varepsilon\]
for every $g \in U$. Taking the supremum on $U$ and taking the limit in the preceding inequality we get
\[  \lim_\cU \sup_{g \in U} \|\pi_i(g) x_i- x_i\| \leq (2+M_0)\varepsilon.\]
This proves (\ref{item:restriction3}).
\end{proof}
\begin{lem}\label{lem:ultraproduct_convergence_on_compactly_supported_measures}
Let $x_k \in E_\cU$, $1 \leq k \leq n$ be a finite family with representative $(x_{k,i})_i$, and $Q \subset G$ a compact subset. If $M(Q)_1$ denotes the set of all complex measures supported in $Q$ and with total variation $\leq 1$, we have
\[\lim_\cU \sup_{m_1,\dots,m_n \in M(Q)_1} |\|\sum_{k=1}^n \pi_i(m_k) x_{k,i}\| - \| \sum_{k=1}^n \pi_\cU(m_k) x_{k} \| | =0.\]

Also, for every $x = (x_i)_i \in E_\cU$ and every compactly supported complex measure on $G$, $\int \pi_\cU(g) x dm(g) = (\int \pi_i(g) x_i dm(g))_\cU$. In particular $\pi_\cU(f) = (\pi_i(f))_\cU$ for every $f \in C_c(G)$.
\end{lem}
\begin{proof}
The formula $\pi_{\cU}(m_k)x_k = (\pi_i(m_k) x_{k,i})_\cU$ is obvious by linearity when $m_k$ is a finitely supported measure by definition of $\pi_{\cU}$, and it implies \[\lim_\cU \|\sum_{k=1}^n \pi_i(m_k) x_{k,i}\|= \| \sum_{k=1}^n \pi_\cU(m_k) x_{k} \|.\] The convergence is easily seen to be uniform among all measures of total variation $\leq 1$ and support contained in a fixed finite subset of $G$. We will reduce to finitely supported measures with the help of the preceding lemma. Let $\varepsilon>0$. By (\ref{item:restriction3}) in Lemma \ref{lem=characterization_XU} there is a neighbourhood $U$ of the identity such that 
\[\lim_\cU \sup_{g \in U} \|\pi_i(g) x_{k,i} - x_{k,i}\| \leq \varepsilon\]
for every $k \in \{1,\dots,n\}$. By compactness of $Q$ there exists a finite subset $Q'=\{g_1,\dots,g_t\} \subset Q$ such that $Q \subset \cup_{s=1}^t g_s U$. Let us extract a partition $Q=Q_1\cup \dots \cup Q_t$ with $Q_s \subset g_s U$. If we denote $M=\sup_{(\pi,E) \in \cR} \sup_{g \in Q} \|\pi_i(g)\|$, we deduce that
\[ \lim_\cU \sup_{g \in Q_s} \|\pi_i(g) x_{k,i} - \pi_i(g_s) x_{k,i}\| \leq M \varepsilon.\]
Let $m_k$ be a complex measure supported in $Q$ and with total variation less than $1$, and define $m'_k = \sum_p m_k(Q_p) \delta_{g_p}$. It is a signed measure supported in $Q'$ and with total variation less than $1$. We have
\[ \| \pi_i(m_k) x_{k,i}-\pi_i(m'_k) x_{k,i}\| \leq \max_k \sup_{g \in Q_k} \|\pi_i(g) x_{k,i} - \pi_i(g_k) x_{k,i}\|,\]
and therefore we get (using also the same estimates for $\pi_{\cU}$)
\begin{multline*}\lim_\cU \sup_{m_1,\dots,m_n \in M(Q)_1} |\|\sum_k \pi_i(m_k) x_{k,i} \| - \|\sum_k \pi_{\cU}(m_k)x_k\| | \\ \leq \lim_\cU \sup_{m'_k \in M(Q')_1} |\| \sum_k \pi_i(m_k') x_{k,i}\| - \| \sum_k \pi_{\cU}(m_k')x_k\| | + 2nM\varepsilon = 2nM\varepsilon.\end{multline*}
The last equality is because $Q'$ is finite. By taking the limit as $\varepsilon$ goes to $0$ we get
\[\lim_\cU \sup_{m_1,\dots,m_n \in M(Q)_1} |\|\sum_{k=1}^n \pi_i(m_k) x_{k,i}\| - \| \sum_{k=1}^n \pi_\cU(m_k) x_{k} \| | =0.\]

We can now move to the second part. Take $m'$ a finitely supported measure such that $\| \pi_\cU(m) x - \pi_\cU(m')x \|<\varepsilon$. Then since $\pi_{\cU}(m')x = (\pi_i(m') x_i)_\cU$
\[ \|\pi_\cU(m) x - (\pi_i(m) x_i)_\cU\| \leq \|\pi_\cU(m) x - \pi_\cU(m') x\| + \|( \pi_i(m') x_i-\pi_i(m) x_i)_\cU\| \leq 2\varepsilon,\]
because $\|( \pi_i(m') x_i-\pi_i(m) x_i)_\cU\| = \|\pi_\cU(m) x - \pi_\cU(m') x\|$ by the first part of the lemma with $n=1$. We conclude by taking the limit $\varepsilon \to 0$.
\end{proof}
We will use the following standard Banach algebraic lemma.
\begin{lem} Assume that $\prod_{\cU} E_i$ is reflexive. Then $E_\cU$ is $M$-complemented in $\prod_{\cU} E_i$, where $M$ is the infimum over all neighbourhoods $V$ of the identity in $G$ of $\lim_\cU \sup_{g \in V} \|\pi_i(g)\|$.
\end{lem}
\begin{proof} Let $(f_n)$ be an approximate unit in $C_c(G)$, and $Q_n$ the support of $f_n$. Then $\|\pi(f_n)\| \leq \lim_\cU \sup_{g \in Q_n} \|\pi_i(g)\|$. Therefore $\limsup_n \|\pi(f_n)\| \leq M$. Since $\prod_{\cU} E_i$ is reflexive, the balls in $B(\prod_\cU E_i)$ are compact for the weak operator topology. So there exists $P \in B(\prod_\cU E_i)$ an accumulation point, in the weak operator topology, of the net $(\pi(f_n))$. It has norm $\|P\|\leq M$. The image of $P$ is contained in $E_\cU$ because this is the case for $\pi(f_n)$ for all $n$. Moreover, the restriction of $P$ to $E_\cU$ is the identity because (Lemma \ref{lem=characterization_XU}) $\pi(f_n) x$ converges in norm to $x$ for every $x \in E_\cU$. Therefore $P$ is a projection in $E_\cU$ of norm $M$, as requested.
\end{proof}
We deduce the following.
\begin{prop}\label{prop:ultraproduct_Lp} Let $1<p<\infty$. Assume that $E_i$ is an $L_{p_i}$-space with $\lim_\cU p_i = p$, and that $\lim_{\cU} \|\pi_i(g)\|=1$ uniformly on compact subsets of $G$. Then $(\pi_\cU,E_\cU)$ is an isometric representation on an $L_p$-space.
\end{prop}
\begin{proof} It is clear that $\pi_\cU$ is an isometric representation. It is well-known that $\prod_\cU E_i$ is an $L_p$ space. For example the proof of \cite[Theorem 3.3 (ii)]{Heinrich} applies without a change. In particular it is reflexive, and by the preceding lemma, $E_\cU$ is a $1$-complemented subspace of an $L_p$ space, and therefore is isometric to an $L_p$ space by \cite{tzafriri}.
\end{proof}
\subsection{Ultraproducts of affine actions}
Let $(\sigma_i)_{i \in I}$ be a family of affine actions of $G$ on a Banach space $E_i$, with linear part $\pi_i$ and translation part $b_i \colon G \to E_i$. This means that $\sigma_i(g) x = \pi_i(g) x + b_i(g)$ for all $x \in E_i$. We assume that $\sup_i \|\pi_i(g)\|$ is bounded on compact subsets of $G$. Let $\cU$ be an ultrafilter on $I$. We wish to define the ultraproduct of $\sigma_i$ as the continuous affine action with linear part the ultraproduct of $\pi_i$ and translation part $b(g) = (b_i(g))_{\cU} \in \prod_{\cU} E_i$. An obvious necessary condition is that $(b_i(g))_{i \in I}$ is bounded and that $b$ is continuous at $0$. The following proposition shows that this is not far from being sufficient. 

\begin{prop}\label{prop:ultraproduct_affine} Let $\pi_i$ be a family of representations of $G$ on Banach spaces $E_i$, with $\sup_i \|\pi_i(g)\|$ bounded on compact subsets of $G$, and let $(\sigma_i)_{i}$ be affine actions of $G$ on a Banach space $E_i$, with linear part $\pi_i$ and translation part $b_i \colon G \to E_i$. Assume that the cocycles $b_i$ are pointwise bounded~:
\[\forall g, \sup_{i} \|b_i(g)\|<\infty\]
and equicontinuous at the identity of $G$~:
\begin{equation}\label{eq:equicontinuity_cocycles} \forall \varepsilon>0, \exists U\subset G \textrm{ neighbourhood of the identity}, \sup_{i}\sup_{g \in U} \|b_i(g)\|<\varepsilon.\end{equation}
Let $\cU$ be an ultrafilter on $I$. There is a continuous affine action $\sigma_\cU$ of $G$ on $E_\cU$  with linear part $\pi_\cU$ and translation part $b(g) = (b_i(g))_\cU$.

Moreover for every $ x=(x_i)_\cU \in E_\cU$, we have
\begin{equation}\label{eq:convergence_delta_sigma_ultraproducts}\delta_S^{\sigma_\cU}( x) = \lim_U \delta_S^{\sigma_i}( x_i ).\end{equation}
\end{prop}
\begin{proof} From the assumption $\sup_{i} \|b_i(g)\|<\infty$, the element $b(g) = (b_i(g))_\cU$ is well-defined in $\prod_\cU E_i$, and it is clear that $b(g)$ satisfies the cocycle relation $b(gh) = b(g)+ (\pi_i(g))_\cU b(h)$, so the only things that deserve a proof are that $b(g)$ belongs to $E_\cU$ and that $b$ is continuous.

Fix $g \in G$ and $\varepsilon>0$. Let $U$ be neighbourhood of the identity in $G$. From the cocycle relation we deduce that $\pi_i(h)b_i(g) - b_i(g) = \pi_i(g) b_i(g^{-1} h g) - b_i(h)$ and that 
\[ \sup_{h \in U} \| \pi_i(h) b_i(g) - b_i(g)\| \leq  (1+\|\pi_i(g)\|) \sup_{ h \in U \cup g^{-1} U g} \|b_i(h)\|.\]
By \eqref{eq:equicontinuity_cocycles} there is a choice of $U$ such that the preceding is less than $\varepsilon$. Since $\varepsilon$ is arbitrary Lemma \ref{lem=characterization_XU} implies that $b(g) \in E_\cU$.
 
The continuity of $b$ at the identity is immediate from \eqref{eq:equicontinuity_cocycles}. From the relation $b(g) - b(g_0) = \pi_\cU(g_0) b(g_0^{-1} g)$ the continuity of $b$ at the identity implies the continuity of $b$ at every point $g_0 \in G$. This concludes the proof that $b \in Z^1(G;\pi_\cU)$.

By writing 
\[\delta_S^{\sigma_\cU}( x) = \sup_{s \in S} \|\pi_\cU(s) x+b(s)\| \textrm{ and }  \delta_S^{\sigma_i}( x_i ) = \sup_{s \in S} \|\pi_i(s) x+b_i(s)\|,\]
the last assertion follows from Lemma \ref{lem:ultraproduct_convergence_on_compactly_supported_measures} and the equicontinuity of the maps $b_i$ that we just established.
\end{proof}

\subsection{Topology and ultraproducts}
We now characterize (when ultraproducts make sense) the topology defined in \S \ref{subsection:def_topolofy} in terms of representation ultraproducts.  For ultraproducts to make sense, we assume in this part that $\cR$ is a set of equivalence classes of Banach space representations of $G$ such that $\sup_{(\pi,E) \in \cR} \|\pi(g)\|$ is bounded on compact subsets of $G$.

\begin{thm}\label{thm:topology_ultraproducts} The closure of a subset $A \subset \cR$ coincides with all equivalence classes belonging to $\cR$ of subrepresentations of an ultraproduct $(\pi_\cU,E_\cU)$ of representations in $A$.
\end{thm}
\begin{proof} We have two inclusions to prove. First assume that $(\pi,E) \in \cR$ is equivalent to a subrepresentation of an ultraproduct $(\pi_\cU,E_\cU)$ of a family $(\pi_i,E_i)_{i \in I} \in A^I$. Let $u \colon E \to E_\cU$ the corresponding $G$-equivariant isometry. Let $x_1,\dots,x_n \in E$, $Q \subset G$ a compact subset. Pick representatives $(x_{k,i})_{i} \in \prod_I E_i$ of $u(x_k) \in E_\cU$. Define a linear map $u_i\colon  F \to E_i$ by $u_i(x_k)=x_{i,k}$. Lemma \ref{lem:ultraproduct_convergence_on_compactly_supported_measures} implies that for every $\varepsilon>0$,  $(\pi_i,E_i)$ belongs to $W_{x_1,\dots,x_n,Q,\varepsilon}(\pi,E)$ for $\cU$-almost every $i$. In particular, every neighbourhood of $(\pi,E)$ intersects $A$, thereby proving the first inclusion. 

Let us move to the second inclusion. Assume that $(\pi,E)$ belongs to the closure of $A$. Consider $I$, the set of all triples $(F,Q,\varepsilon)$ where $F=(x_1,\dots,x_n)$ is a finite sequence of elements of $E$, $Q \subset G$ is compact containing a neighbourhood of the identity and $\varepsilon>0$. It is a directed set for the order $(F,Q,\varepsilon) \leq (F',Q',\varepsilon')$ when $F \subset F'$,  $Q \subset Q'$ and $\varepsilon\geq \varepsilon'$. For every $i=(F,Q,\varepsilon) \in I$, there is $(\pi_i,E_i) \in A \cap W_{F,Q,\varepsilon}(\pi,E)$. Let $u_i(x_{1}),\dots,u_i(x_n) \in E_{i}$ be a witness of this inclusion. Note in particular that (since $e \in \mathrm{Int}(Q)$) taking $m_k\in \{0,\delta_e\}$ in the definition of $W_{F,Q,\varepsilon}(\pi,E)$ (see Remark \ref{rem:equivalence_WxQe}) we have
\begin{equation}\label{eq:u_i_a_linear} | \| u_i(x)\| - \|x\| \|  \leq \varepsilon \textrm{ for all }x \in F,\end{equation}
and taking $m_k$ a multiple of $\delta_e$
\begin{equation}\label{eq:u_i_a_isometric} \| u_i(x+\lambda y) - u_i(x) - \lambda u_i(y) \|  \leq \varepsilon \max(1,|\lambda|)\end{equation}
for all $x,y \in F$ and $\lambda \in \C$ such that $x+\lambda y \in F$. Similarly, taking $m_k \in \{0,\frac{1}{2}(\delta_e - \delta_g), g \in \mathrm{Int}(Q)\}$ we get that for all $U \subset \mathrm{Int}(Q)$, 
\begin{equation}\label{eq:u_i_equicontinuous} \sup_{g \in U} \| \pi_{i}(g) u_i(x) -u_i(x)\| \leq 2 \varepsilon + \sup_{g \in U} \|\pi(g) x-x\| \textrm{ for all }x \in F.\end{equation}
Finally taking $m_k \in \{0,\delta_g,g \in \mathrm{Int}(Q)\}$ we have
\begin{equation}\label{eq:u_i_a_equivariant} \| \pi_{i}(g) u_i(x) - u_i(\pi(g)x)\| \leq \varepsilon\end{equation}
for all $x \in F$ and $g \in \mathrm{Int}(Q)$ such that $\pi(g) x \in F$. We extend $u_i$ to a map $u_i \colon E \to E_{\alpha(i)}$ by setting $u_i(x) = 0$ is $x \notin F$. Finally let $\cU$ be a cofinal ultrafilter on $I$~: $\{i \in I, i \geq i_0\} \in \cU$ for all $i_0 \in I$. Consider the ultraproduct $(\pi_{\cU},E_{\cU})$. For every $x \in X$, define $u(x) = (u_i(x))_{\cU} \in \prod_\cU E_{i}$. Since $\cU$ is cofinal, \eqref{eq:u_i_a_linear} and \eqref{eq:u_i_a_isometric} imply that $u$ is linear isometric, \eqref{eq:u_i_equicontinuous} and Lemma  \ref{lem=characterization_XU} imply that $u$ takes its values in $E_\cU$, and \eqref{eq:u_i_a_equivariant} implies that it satisfies $u(\pi(g) x) = \pi_\cU(g) u(x)$ for all $x \in E$ and $g \in G$. Therefore, $u$ realizes $(\pi,E)$ as a subrepresentation of $(\pi_\cU,E_\cU)$. This shows the second inclusion.
\end{proof}

We can now give the
\begin{proof}[Proof of Proposition \ref{prop:strong_neigh_discrete}] As in the proof Proposition \ref{prop:strong_neigh_Em}, we have to show that every accumulation point $(\pi,E)$ of a sequence $(\pi_n,E_n) \in \cG(\cE,m+o(1) \ell)$ belongs to $\cG(\cE,m)$. For every $n$, realize $(\pi_n,E_n)$ as a subrepresentation of $(\pi'_n,E'_n)$ with $E'_n \in \cE$ and $\|\pi'_n(g)\| \leq e^{m(g)}$. By Theorem \ref{thm:topology_ultraproducts}, $(\pi,E)$ is a subrepresentation of an ultraproduct $(\pi_\cU,E_\cU)$ of a family $(\pi_{n_i},E_{n_i})_{i \in I}$ with $\lim_\cU n_i=\infty$, which is itself a subrepresentation of the ultraproduct $(\pi'_\cU, E'_\cU)$ of $(\pi'_{n_i},E'_{n_i})_{i \in I}$. Then $\|\pi'_\cU(g)\| \leq \lim_\cU \|\pi'_{n_i}(g)\| \leq e^{m(g)}$ for every $g$. Moreover, since $G$ is discrete, $E'_\cU = \prod_\cU E_{n_i}$, which belongs to $\cE$ because $\cE$ is stable by ultraproducts. So $(\pi,E) \in \cG(\cE,m)$.
\end{proof}

\begin{proof}[Proof of Proposition \ref{prop:strong_neigh_Lp}] The same proof applies, except that instead of using that $G$ is discrete, use Proposition \ref{prop:ultraproduct_Lp}.
\end{proof}

\section{Applications}\label{sec:applications}

\subsection{Choice of $\cR$}\label{subsection:R} 
In the previous section $\cR$ could be an arbitrary set of equivalence classes of Banach space representations, but in the applications we need $\cR$ to be large enough. For this we fix $M>1$ and take for $\cR$ the set of all equivalence classes of Banach space representations $(\pi,E)$ on a separable Banach space and such that $\sup_{g \in S} \|\pi(g)\| \leq M$ for all $g \in G$. 

The reason why we impose a bound on $\|\pi(g)\|$ is that this ensures that ultraproducts of representations in $\cR$ make sense. The reason why we have to bound the dimensions of the spaces in $\cR$ is because otherwise $\cR$ would not be a set. The separability is enough for our purposes because $\cR$ is stable by ultraproducts in the following sense~: for every ultraproduct $(\pi_\cU,E_\cU)$ of a family $(\pi_i,E_i)_i\in \cR^I$ and every continuous function $f \colon G \to E_\cU$, there is a $\pi_\cU(G)$-invariant closed subspace $E \subset E_\cU$ containing $f(G)$ such that the equivalence class of $(\pi,E)$ belongs to $\cR$.

\subsection{Proof of Theorem \ref{thm:Kazproj+H1_open}}
We now prove the following precise form of Theorem \ref{thm:Kazproj+H1_open}, in the general setting of locally compact compactly generated groups.
\begin{thm}\label{thm:Kazproj} Let $\cF \subset \cR$. Assume that $\cC_{\cF}(G)$ has a Kazhdan projection and that $H^1(G,\pi)=0$ for all $(\pi,E) \in \cF$. Then there is a strong neighbourhood $\cF'$ of $\cF$ such that $\cC_{\cF'}(G)$ has a Kazhdan projection and that $H^1(G,\pi)=0$ for all $(\pi,E) \in \cF'$.
\end{thm}
\begin{rem} As we shall see in \S \ref{subsec:Lp}, it is not true that if one moreover assumes that $\cC_{cF}(G)$ has a central Kazhdan projection, then $\cC_{\cF'}(G)$ has a Kazhdan projection.
\end{rem}
\begin{proof}
By Theorem \ref{thm=local_characterization_of_Kazhdan_constant_nondiscrete}, there is a a compactly supported measure $m$ with $\int 1 dm=1$ such that $\delta_S^\sigma( \sigma(m) x) \leq \frac 1 2 \delta_S^\sigma(x)$ for every affine action $\sigma$ on $E$ with linear part in $\cF$ and every $x \in E$. We shall find another measure $m'$ and a strong neighbourhood $\cF'$ of $\cF$ such that $\delta_S^\sigma( \sigma(m') x) \leq \frac 1 2 \delta_S^\sigma(x)$ for every affine action $\sigma$ with linear part $(\pi,E)$ in $\cF'$ and every $x \in E$. By the converse direction in Theorem \ref{thm=local_characterization_of_Kazhdan_constant_nondiscrete} this will imply that $\cC_{\cF'}(G)$ has a Kazhdan projection and that $H^1(G,\pi)=0$ for all $(\pi,E) \in \cF'$, and prove the first part of the theorem.

The proof is particularily simple in the case when the group is discrete, and in that case we can take $m'=m$ (and replace $\frac 1 2$ by $\frac 1 {\sqrt 2}$, say). In the general case we fix a nonnegative function $f_0 \in L_1(G)$ with compact support $Q_0$ and $\int f_0=1$ which will be used for regularization (if $G$ is discrete just take $f_0 = \delta_e$). By Lemma \ref{lem=trivialite} there is $C>0$ such that 
\begin{equation}\label{eq:f_0_does_not_increase_delta_too_much}\delta^\sigma_S(\sigma(f_0)x) \leq C \delta^\sigma_S(x)\end{equation} for every affine action $\sigma$ on a Banach space $E$ with linear part in $\cR$ and every $x \in E$. Take $k\in \N$ such that $\frac C {2^k} < \frac 1 2$, and denote $m' = m^{\ast k} \ast f_0$. Define $\cF'$ as the set of representation classes $(\pi,E) \in \cR$ such that $\delta_S^\sigma( \sigma(m') x) \leq \frac 1 2 \delta_S^\sigma(x)$ for every affine action $\sigma$ with linear part $(\pi,E)$ and every $x \in E$.

To prove that $\cF'$ is a strong neighbourhood of $\cF$ we use the characterization in Lemma \ref{lem:characterization_strong_neigh}: we consider a net $(\pi_i,E_i)_{i \in I}$ contained in $\cR \setminus \cF'$, and we have to construct an accumulation point of this net which does not belong to $\cF$. By definition of $\cF'$, for every $i \in I$, there is an affine action $\sigma_i$ of $G$ on $E_i$ with linear part $\pi_i$ and $x_i \in E_i$ such that 
\[ \delta_S^{\sigma_i}(\sigma(m') x_i) > \frac 1 2 \delta_S^{\sigma_i}( x_i).\]
By normalizing we can assume that $\delta_S^{\sigma_i}( x_i)=1$. The formula \[b_i(g) = \sigma_i(g) \sigma_i(f_0) x_i -\sigma_i(f_0)x_i\] defines a cocycle with values in $\pi_i$, cohomologous to the cocycle $g \mapsto \sigma_i(g) 0$. By \eqref{eq:f_0_does_not_increase_delta_too_much}, it satisfies $\sup_{g \in S} \|b_i(g)\| \leq C$, and therefore by the cocycle relation we have $\sup_{i} \|b_i(g)\|<\infty$ for every $g \in G$. Moreover, we have
\[ \|b_i(g)\| \leq \| \lambda_g f_0 - f_0\|_{L_1(G)} \sup_{h,h' \in Q_0} \| \sigma_i(gh) x_i-\sigma_i(h') x_i\|.\]
Since $\lim_{g \to e} \| \lambda_g f_0 - f_0\|_{L_1(G)}=0$, this implies that the $b_i$ are equicontinuous on the neighbourhood of $e$. Take $\cU$ a cofinal ultrafilter on $I$. By Propostion \ref{prop:ultraproduct_affine}, $b(g) = (b_i(g))_\cU$ is a cocycle with values in $(\pi_\cU,E_\cU)$, and defines an affine action $\sigma$. Since $G$ is separable, there is a separable closed subspace $E \subset E_\cU$ that is invariant under $\pi$ and $\sigma$. Then $(\pi_\cU,E) \in \cR$ and by Theorem \ref{thm:topology_ultraproducts} and the cofinality of $\cU$, $(\pi_\cU,E)$ is an accumulation point of the net $(\pi_i,E_i)_i$. On the one hand we have $\delta_S^\sigma(0) = \sup_{g \in S} \|b(g)\| \leq C$, and on the other hand by \eqref{eq:convergence_delta_sigma_ultraproducts} in Proposition \ref{prop:ultraproduct_affine},
\[ \delta_S^\sigma(\sigma(m^k) 0) = \lim_\cU  \delta_S^{\sigma_i}(\sigma_i(m^k)\sigma_i(f_0) x_i) = \lim_\cU  \delta_S^{\sigma_i}(\sigma_i(m') x_i)  \geq \frac 1 2.\]
By the definition of $k$ we therefore have  $\delta_S^\sigma( \sigma(m^k) 0) > 2^{-k} \delta_S^\sigma(0)$. This implies that $(\pi_\cU,E) \notin \cF$ and concludes the proof of the first half of the theorem.
\end{proof}

\subsection{Fixed point and Kazhdan projections}\label{subsection:fixed_point_KP}
The next lemma, a form of which has been proved by Masato Mimura, gives an improvement on Guichardet's argument that $H^1(G;\pi)=0$ implies that $E/E^\pi$ does not have almost invariant vectors. We provide a proof for completeness.
\begin{lem}\label{lem:FE_implies_uniform_spectral_gap} Let $\cF \subset \cR$ be closed with the property that for every $(\pi,E) \in \cF$ and every continuous affine action of $G$ on $E$ with linear part $\pi$ has a fixed point. There exists $\varepsilon>0$ such that $\delta_S^\pi(x) \geq \varepsilon \|x\|_{E/E^\pi}$ for every $(\pi,E) \in \cF$ and $x \in E$.
\end{lem}
\begin{proof}
If $(\pi,E)$ is a Banach space representation of $G$, we denote by $\varepsilon(\pi,S)$ the associated Kazhdan constant, \emph{i.e.} the best (=largest) $\varepsilon$ such that $\delta_S^\pi(x) \geq \varepsilon \|x\|_{E/E^\pi}$ for every $x \in E$. If the lemma was not true, there would exist a sequence $(\pi_n,E_n) \in \cE$ such that $\lim_n \varepsilon(\pi_n,S)=0$. Denote $\varepsilon_n = \varepsilon(\pi_n,S)$. As recalled above, $\varepsilon_n$ is strictly positive. Let $\cU$ be a free ultrafilter on $\N$ and $(\pi_\cU,E_\cU)$ the ultraproduct of $(\pi_n,E_n)$. Let $x_n \in E_n$ such that $\delta_S^\pi(x_n)=1$ and $\|x_n\|_{E_n/E_n^{\pi_n}} > \frac{1}{2\varepsilon_n}$. Let $f_0 \in C_c(G)$ nonnegative with $\int f_0=1$ and $y_n = \pi(f_0) x_n $. By Lemma \ref{lem=trivialite} there is $C>0$ such that $\|x_n - y_n\| \leq C$ for all $n$, so that 
\[ \|y_n\|_{E_n/E_n^{\pi_n}}  \geq \|x_n\|_{E_n/E_n^{\pi_n}} - \|y_n-x_n\|_{E_n} \geq \frac{1}{3\varepsilon_n}\]
for all $n$ large enough. Moreover, if we define $b_n(g) = \pi(g) y_n - y_n$,  then as in the proof of Theorem \ref{thm:Kazproj} $b(g) := (b_n(g))_\cU$ defines an affine action $\sigma$ with linear part $(\pi_\cU,E_\cU)$ and there is a closed subspace $E \subset E_\cU$ that is invariant under $\pi_\cU$ and $\sigma$ such that $(\pi_\cU,E) \in \cR$. By Theorem \ref{thm:topology_ultraproducts} $(\pi_\cU,E)$ belongs to $\cF$ because $\cF$ is closed. This affine action therefore has a fixed point $z = (z_n)_\cU \in E$, so that the class of $(b_n(g))$ in the ultraproduct coincides with the class of $(\pi(g) z_n - z_n)_n$ for some $z=(z_n)_\cU \in E$. By proposition \ref{prop:ultraproduct_affine}, we get $\lim_{\cU} \delta_S^{\pi_n}(y_n - z_n)=0$, whereas $\| y_n - z_n \|_{E_n/E_n^\pi} \geq \|y_n\|_{E_n/E_n^\pi}- \|z_n\| \geq \frac{1}{4\varepsilon_n}$ for all $n$ large enough. This contradicts the definition of $\varepsilon_n$. 
\end{proof}
We can now prove Proposition \ref{prop:FE_implies_KazhdanProj}
\begin{proof}[Proof of Proposition \ref{prop:FE_implies_KazhdanProj}] If $\cE$ is stable by finite representability, then $\cF$, the set of all isometric representations $(\pi,E) \in \cR$ on a space in $\cE$, is closed (Remark \ref{rem:closure_implies_fin_rep+bound}). By Lemma \ref{lem:FE_implies_uniform_spectral_gap} and \cite{drutunowak} $\cC_{\cF}(G)$ has a Kazhdan projection.

If $\cE$ is only closed by ultraproducts and $G$ is discrete, then the same proof applies to $\cF =\cG(\cE,0)$ (see Proposition \ref{prop:strong_neigh_discrete}). If $\cE=L_p$ the same proof applies with $\cF=\cG(L_p,0)$ (Proposition \ref{prop:strong_neigh_Lp}).
\end{proof}

\subsection{Application to fixed point properties}\label{subsection:app_fixed_point}
Recall that if $\cE$ is a class of Banach spaces, we denote by $\cF(\cE,m)$ the set of all equivalence classes of representations $(\pi,E) \in \cR$ such that $E$ is isomorphic to a space in $\cE$ and such that $\|\pi(g)\|_{B(E)}\leq e^{m(g)}$ for all $g \in G$. 

\begin{cor}\label{cor:direct_cor}
Let $\cE$ be a class of Banach spaces closed under finite representability and $m \colon G\to (0,\infty)$. If $\cC_{\cF(\cE,m)}(G)$ has a Kazhdan projection and every affine action with linear part in $\cF(\cE^{N,\varepsilon},m)$ has a fixed point, then there exists $N\in \N,\varepsilon>0$ such that the same is true for $\cF(\cE^{N,\varepsilon},m+\varepsilon \ell)$. 

If moreover $m$ is symmetric and $\cE$ is stable by duality, then there exists $N,\varepsilon$ such that the Kazhdan projection in $\cC_{\cF(\cE^{N,\varepsilon},m+\varepsilon \ell)}(G)$ is central.
\end{cor}
\begin{proof} By Theorem \ref{thm:Kazproj} there is a strong neighbourhood $\cF'$ of $\cF(\cE,m)$ such that $\cC_{\cF'}(G)$ has a Kazhdan projection and every affine action with linear part in $\cF'$ has a fixed point. By Proposition \ref{prop:strong_neigh_Em} $\cF'$ contains $\cF(\cE^{N,\varepsilon},m')$ for some $\varepsilon,N$, with $m'=m+\varepsilon\ell$. 

It remains to prove that, if $N,\varepsilon$ is replaced by some $N'>N,\varepsilon'<\varepsilon$, the Kazhdan projection can be taken central if $\cE$ is stable by duality. By Corollary \ref{cor:KP_automatically_central} it is enough to show that for some $N',\varepsilon'$, $\cF(\cE^{N',\varepsilon'},m')$ is contained in a weakly self-adjoint subset of $\cF'$. But by Lemma \ref{lem:duality_finite_representability} there is $N' \geq N$ and $\varepsilon'<\varepsilon$ such that $X^*$ belongs to $\cE^{N,\varepsilon}$ for all $X \in \cE^{N',\varepsilon'}$. Since $m$ and $m'$ are symmetric, this implies that the set of representations in $\cF(\cE^{N,\varepsilon},m')$ such that all separable subrepresentations of its dual representation belong to $\cF(\cE^{N,\varepsilon},m')$ contains $\cF(\cE^{N',\varepsilon'},\varepsilon')$. But this set is clearly weakly self-adjoint. This concludes the proof.
\end{proof}

We can now prove the following corollaries, mentioned in the introduction.
\begin{cor}\label{cor:T_implies_robustT} 
If $G$ has property (T), then there exists $\varepsilon>0$ and a central Kazhdan projection in  $\cC_{\cF(\cE(\varepsilon),\varepsilon \ell)}(G)$, and every affine action with linear part in $\cF(\cE(\varepsilon),\varepsilon \ell)$ has a fixed point, where $\cE(\varepsilon)$ is the class of all Banach spaces $E$ satisfying
\begin{equation}\label{eq=X-epsilon-close} \frac 1 2  \left(\|x+y\|^2 + \|x-y\|^2 \right) \leq (1+\varepsilon) \left(\|x\|^2 + \|y\|^2 \right) \ \ \forall x,y \in E.\end{equation}

\end{cor}
\begin{proof}
This is Corollary \ref{cor:direct_cor} for $\cE$ the class of Hilbert spaces (recall Proposition \ref{prop:strong_neigh_Hilb}).
\end{proof}
\begin{cor}\label{cor:FE_implies_robustFE} Let $\cE$ be a class of superreflexive Banach spaces closed under finite representability. If $G$ has (F$_{\cE}$) then there exists $N\in \N$ and $\varepsilon>0$ such that $\cC_{\cF(\cE^{N,\varepsilon},\varepsilon \ell)}(G)$ has a Kazhdan projection and every affine action with linear part in $\cF(\cE^{N,\varepsilon},\varepsilon \ell)$ has a fixed point. The Kazhdan projection is central if $\cE$ is stable by duality.
\end{cor}
\begin{proof}
Combine Proposition \ref{prop:FE_implies_KazhdanProj} and Corollary \ref{cor:direct_cor}.\end{proof}
\begin{rem}
The (proof of the) preceding corollary says that if a class of superreflexive Banach spaces $\cE$ is stable by finite representability then (F$_{\cE}$) implies robust property (T) with respect to $\cE$ (and more generally to $\cE^{N,\varepsilon}$ for some $N,\varepsilon$), see \S \ref{subsection:examples} for the definition of Oppenheim's robust property (T). Oppenhein proved that the converse holds for every set $\cE$ such that every space $X \in \cE$ there is $X' \in \cE$ isometric to $X \oplus_p \C$ for some $1 \leq p \leq \infty$. Together, this shows that (F$_{\cE}$) is equivalent to robust (T) with respect to $\cE$ if $\cE$ is the class of Hilbert spaces, or the class of spaces $C$-isomorphic to Hilbert spaces, or (for some $1<p<\infty$) the class of subspaces of $L_p$ spaces, or the class of subquotients of $L_p$ spaces... We can also replace $L_p$ spaces by non-commutative $L_p$ spaces because non-commutative $L_p$ spaces are closed under ultraproducts \cite{raynaud}. Corollary \ref{cor:FLp_implies_robustFLp} will also imply that (F$_{L_p}$) is equivalent to robust (T) with respect to $L_p$ spaces. 
\end{rem}
Recall that one says that $G$ has ($\overline{\mathrm{F}}_{\cE}$) if every affine action on a space in $\cE$ whose linear part is a uniformly bounded representation has a fixed point.

\begin{cor}\label{cor:FE} Let $\cE$ be a class of superreflexive Banach spaces closed under finite representability. The following are equivalent~:
\begin{enumerate}
\item\label{item:FbarE} $G$ has ($\overline{\mathrm{F}}_{\cE}$).
\item\label{item:almoststrongTFEclose} For every $C>0$, there exists $N \in \N$ and $\varepsilon >0$ such that $\cC_{\cF(\cE^{N,\varepsilon},\varepsilon \ell +C)}(G)$ has a Kazhdan projection and $H^1(G;\pi)=0$ for every $(\pi,E) \in \cF(\cE^{N,\varepsilon},\varepsilon \ell +C)$.
\item\label{item:almoststrongTFE} For every $C>0$, there exists $\varepsilon >0$ such that $\cC_{\cF(\cE,\varepsilon \ell +C)}(G)$ has a Kazhdan projection and $H^1(G;\pi)=0$ for every $(\pi,E) \in \cF(\cE,\varepsilon \ell +C)$.
\item\label{item:almoststrongT} (If $\cE$ contains a space of infinite dimension) For every $C>0$, there exists $\varepsilon >0$ such that $\cC_{\cF(\cE,\varepsilon \ell +C)}(G)$ has a Kazhdan projection.
\end{enumerate}
In that case, and if $\cE$ is stable by duality, then $\cC_{\cF(\cE,\varepsilon \ell +C)}(G)$ has a central Kazhdan projection.
\end{cor}
\begin{proof} (\ref{item:FbarE}) $\implies$ (\ref{item:almoststrongTFEclose}) is Corollary \ref{cor:direct_cor}. The implications (\ref{item:almoststrongTFEclose}) $\implies$ (\ref{item:almoststrongTFE})$\implies$(\ref{item:almoststrongT}) and (\ref{item:almoststrongTFE})$\implies$ (\ref{item:FbarE}) are obvious. 

Assume (\ref{item:almoststrongT}). By an argument of Lafforgue \cite[\S 5.3]{lafforguefastfourier}, this implies that $H^1(G;\pi)=0$ for every uniformly bounded representation $(\pi,E)$ with $E$ isomorphic to a hyperplane in a space in $\cE$. So (\ref{item:FbarE}) is a consequence of the following claim: every separable space $E$ in $\cE$ is isomorphic to a hyperplane in another space $E'$ in $\cE$. If $E$ is finite dimensional this is obvious because we assumed that $\cE$ contains a space of infinite dimension, and in particular a subspace of dimension $\mathrm{dim}(E)+1$. Otherwise, let $E_n \subset E$ be an increasing sequence of finite dimensional subspaces such that $\cup_n E_n$ is dense in $E$. Let $\cU$ be a cofinal ultrafilter in $\N$. Then $E$ is isometric to a subspace of $\prod_{\cU} E_n$ by sending $x \in \cup_n E_n$ to $(1_{x \in E_n} x)_\cU$ and extending by continuity. It is a strict subspace because $\prod_{\cU} E_n$ is not separable, so if $x \in \prod_{\cU} E_n \setminus E$, we have that the linear span of $x$ and $E$ belongs to $\cE$ and contains $E$ as a hyperplane. 
\end{proof}
The point (\ref{item:almoststrongT}) is almost strong property (T) with respect to $\cE$, except on the order of the quantifiers, which should be $\exists \varepsilon,\forall C$ instead of $\forall C, \exists \varepsilon$. This ``small'' difference is a bit unfortunate, because Shalom conjectured that hyperbolic groups do not have ($\overline{\mathrm{F}}_{\ell_2}$), whereas Lafforgue \cite{lafforguestrongt} proved that hyperbolic groups do not have strong property (T) with respect to Hilbert spaces.

\subsection{Application to $L_p$ spaces}\label{subsec:Lp} 
We now prove the following result for property (F$_{L_p}$).
\begin{cor}\label{cor:FLp_implies_robustFLp} Let $1<p<\infty$. The following are equivalent.
\begin{enumerate}
\item\label{item:FLp} $G$ has property (F$_{L_p}$).
\item\label{item:robustTLp+eps} there exists $\varepsilon>0$ such that $\cC_{\cF(\{L_q,q \in [p-\varepsilon,p+\varepsilon]\},\varepsilon \ell)}(G)$ has a Kazhdan projection and every affine action with linear part in $\cF(\{L_q,q \in [p-\varepsilon,p+\varepsilon]\},\varepsilon \ell)$ has a fixed point.
\item\label{item:robustTLp} there exists $\varepsilon>0$ such that $\cC_{\cF(L_p,\varepsilon \ell)}(G)$ has a Kazhdan projection.
\end{enumerate}
In that case, and for all $\varepsilon$ small enough, the projection in (\ref{item:robustTLp+eps}) and (\ref{item:robustTLp}) are central if and only if $G$ has (F$_{L_{p'}}$) where $p' = \frac{p}{p-1}$ is the conjugate exponent of $p$.
\end{cor}
\begin{rem}\label{rem:noncentralKP} This Corollary provides natural examples where there is a Kazhdan projection, but not a central Kazhdan projection. Indeed, consider $\Gamma$ a discrete Gromov-hyperbolic group with property (T) (for example a cocompact lattice in $\mathrm{Sp}(n,1)$, or a suitable random group). It was proved by Bader, Furman, Gelander and Monod  \cite{baderfurmangelandermonod} that, as every group with property (T), $\Gamma$ has (F$_{L_p}$) for every  $1<p\leq 2$. On the other hand, Yu  \cite{yu} proved that every hyperbolic group with property (T) has a proper action on $\ell_{p'}$ for some $1<p \leq 2$, and in particular does not have (F$_{L_{p'}}$). By the above corollary, there is $\varepsilon$ such that $\cC_{\cF(L_p,\varepsilon \ell)}(G)$ has a Kazhdan projection, but not a central Kazhdan projection.
\end{rem}

\begin{proof}  (\ref{item:robustTLp+eps})$\implies$(\ref{item:robustTLp}) is obvious, and (\ref{item:robustTLp})$\implies$(\ref{item:FLp}) is \cite{oppenheim}. (\ref{item:FLp})$\implies$ (\ref{item:robustTLp+eps}) is not formally a consequence of Corollary \ref{cor:FE_implies_robustFE} because $L_p$ spaces are not stable under subspaces. However, for the first part, the same proof works with Proposition \ref{prop:strong_neigh_Em} replaced by Proposition \ref{prop:strong_neigh_Lp}.

For the second part, if $G$ has (F$_{L_p'}$) and (F$_{L_p}$) then by the first part there is also, for all small enough $\varepsilon_1$, a Kazhdan projection for $\cF(\{L_{q'},q \in [p-\varepsilon_1,p+\varepsilon_1]\},\varepsilon_1 \ell)$, which is the dual of $\cF(\{L_{q},q \in [p-\varepsilon_1,p+\varepsilon_1]\},\varepsilon_1 \ell)$. By replacing $\varepsilon$ by $\min(\varepsilon,\varepsilon_1)$, the implication (\ref{item:exists_dualKP})$\implies$(\ref{item:central}) in Proposition \ref{prop:central} implies that the Kazhdan projection for $\cF(\{L_q,q \in [p-\varepsilon,p+\varepsilon]\},\varepsilon \ell)$ is central, and hence also for $\cF(\{L_p,\varepsilon \ell)$. Conversely, assume that  $\cC_{\cF(L_p,\varepsilon \ell)}(G)$ has a central projection for some $\varepsilon>0$, then by Proposition \ref{prop:central} $\cC_{\cF(L_{p'},\varepsilon \ell)}(G)$ has a Kazhdan projection and hence $G$ has (F$_{L_p}$).
\end{proof}

Similarly using Proposition \ref{prop:strong_neigh_discrete}.
\begin{cor}\label{cor:Delorme-Guichardet_Banach} Let $\cE$ be a class of superreflexive Banach spaces stable by ultraproducts. Then, for discrete groups, (F$_{\cE}$) implies robust (T) with respect to $\cE$.
\end{cor}
Recall \cite{oppenheim} that the converse holds if for every $X \in \cE$, there is $1\leq p \leq \infty$ such that $X \oplus_p \C \in \cE$.
\subsection{From compactly generated to compactly presented}\label{subsection:compactlypresented}
We can now state a result, which is an extension to Banach space representations of a result proved for unitary representations by Shalom \cite{shalom} for discrete groups, and Fisher--Margulis \cite{fishermargulis} for locally compact groups. 

A locally compact group $G$ with a compact generating set $S$ is said to be compactly presented if, as an abstract group, $G$ has a presentation with $S$ as a set of generators and with relators of bounded length (this does not depend on $S$).

If $H$ is a quotient of $G$ and $\cF \subset \cR$ we denote $\cF_{[H]}$ the set of all equivalence classes of representations in $\cF$ which factor through $H$.

\begin{thm}\label{thm:Kajproj_comapctly_presented} Assume that $G$ is compactly presented, that $H$ is a quotient of $G$ by a discrete normal subgroup, and let $\cF \subset \cG$ be closed. If $\cC_{\cF_{[H]}}(H)$ has a Kazhdan projection and $H^1(H;\pi)=0$ for every $(\pi,E) \in \cF_{[H]}$, then there is a compactly presented intermediate quotient $G \to H' \to H$ such that $\cC_{\cF_{[H']}}(H')$ has a Kazhdan projection and $H^1(H';\pi)=0$ for every $(\pi,E) \in \cF_{[H']}$.
\end{thm}
\begin{proof} The theorem is proved exactly as Theorem \ref{thm:Kazproj}, so we only give a short sketch. By Theorem \ref{thm=local_characterization_of_Kazhdan_constant_nondiscrete} there is a measure $m$ with $\int 1 dm=1$ and such that 
\begin{equation} \label{eq:deltaSsigma}\delta_S^\sigma( \sigma(m) x) \leq \frac 1 2 \delta_S^\sigma(x)\end{equation} for every affine action $\sigma$ of $H$ on $E$ with linear part in $\cF_{[H]}$ and every $x \in E$. Let $f_0 \in L_1(G)$, $C >0$, $k \in \N$ and $m' = m^{\ast k} \ast f_0$ as in the proof of Theorem \ref{thm:Kazproj}.

Since $G$ is compactly presented, it has a presentation $G= \langle S,R\rangle$ with relations of length $\leq n_0$. For $n \geq n_0$, let $R_n$ be the set of words of length less than or equal to $n$ in the letters $S$ which are trivial in $H$, and define a sequence of intermediate compactly presented intermediate quotients $H_n$ by
\[H_n= \langle S,R_n\rangle.\] 

Assume by contradiction that for every $n$, there is an affine action $\sigma_n$ of $H_n$ on $E_n$ with linear part in $\cF_{[H_n]}$ and $x_n \in E_n$ such that $\delta_S^{\sigma_n}(\sigma(m') x_n)> \frac  1 2 \delta_S^{\sigma_n}(x_n) = \frac 1 2$. Consider the cocycle $b_n(g) = \sigma_n(g) \sigma_n(f_0) x_n - \sigma_n(f_0) x_n$. Let $\cU$ be a cofinal ultrafilter on $\N$. Then there is a separable subaction $\sigma$ of the ultraproduct action which factors through $H$ (because it factors though $H_n$ for all $n$), with linear part in $\cF$ (because $\cF$ is closed) and with translation part $b(g) = (b_n(g))_{\cU}$. Therefore by using successively Proposition \ref{prop:ultraproduct_affine}, \eqref{eq:deltaSsigma} and \eqref{eq:f_0_does_not_increase_delta_too_much} we get
\[ \frac 1 2 \leq \lim_\cU \delta_S^{\sigma_n}(\sigma_n(m')x_n) = \delta_S^\sigma( \sigma(m)^k 0) \leq 2^{-k} \delta_S^\sigma(0) \leq 2^{-k} C,\]
a contradiction with the definition of $k$.
\end{proof}

An example of consequence is the following result. However, as Masato Mimura pointed out to us, there is a more direct and easy proof, that he attributes to Gromov--Schoen \cite{gromov} (see also \cite{stalder}), of this corollary which works without the assumption that $\cE$ is superreflexive. 
\begin{cor}\label{cor:Flp_open_property} Let $\cE$ be a class of superreflexive Banach spaces closed under finite representability. If a locally compact group compactly generated group $H$ has (F$_{\cE}$), then $H$ is the quotient by a discrete normal subgroup of a compactly presented locally compact group with (F$_{\cE}$).
\end{cor}
\begin{rem} For discrete groups, the same conclusion holds under the weaker assumption that $\cE$ is a class of superreflexive Banach spaces closed under ultraproducts.
\end{rem}
\begin{proof} For convenience of notation we only give the proof when $H$ is separable. It was essentially proved in \cite{abels} and rediscovered in \cite{fishermargulis} that there is a compactly presented group $G$ and a continuous surjective group homomorphism $G \to H$ with discrete kernel. Moreover if one follows the proof, $G$ is separable. Take $\cR$ as defined in \S \ref{subsection:R} for this group $G$ and some $M>1$, and let $\cF = \cF(\cE,0)$. It is closed (see Remark \ref{rem:closure_implies_fin_rep+bound}). By Proposition  \ref{prop:FE_implies_KazhdanProj} $\cC_{\cF_{[H]}}(G)$ has a Kazhdan projection and $H^1(G;\pi)=0$ for every $(\pi,E) \in \cF_{[H]}$. By the previous theorem, there is a compactly presented intermediate group $G \to H' \to H$ such that the same holds for $\cF_{[H']}$. In particular $H'$ has (F$_{\cE}$).
\end{proof}

One can imagine other results of this kind. Let us state one that we will use in a forthcoming work with Gomez-Apparicio and Liao.
\begin{cor}\label{cor:strongTcompactly_presented} Let $\cE$ be a class of Banach spaces stable by finite representability and containing an infinite dimensional space. Assume that $G$ is compactly presented, and that $H$ is a quotient by a discrete subgroup such that $H$ has strong property (T) with respect to $\cE$. Then for every $C>0$ there is $s>0$ and a compactly presented intermediate quotient $G\to H' \to H$ such that $\cC_{\cF(\cE,s\ell +C)}(H')$ has a Kazhdan projection, which is self-adjoint if $\cE$ is stable by duality.
\end{cor}
\begin{proof} Since $H$ has strong (T) with respect to $\cE$, there exists $s>0$ such that $\cC_{\cF(\cE,s\ell+C)_{[H]}}(H)$ has a central Kazhdan projection for every $C>0$. By an argument of Lafforgue \cite[\S 5.3]{lafforguefastfourier}, this implies that $H^1(H;\pi)=0$ for every $(\pi,E)$ with $E$ a hyperplane in a space in $\cE$ and with a constant $C$ such that $\|\pi(g)\| \leq e^{s\ell(g)+C}$ for all $g$ in $G$. So the corollary is a consequence of Theorem \ref{thm:Kajproj_comapctly_presented} and of the fact, already proved in the proof of Corollary \ref{cor:FE}, that every separable space $E$ in $\cE$ is isomorphic to a hyperplane in another space $E'$ in $\cE$. 
\end{proof}


\begin{thebibliography}{KM98b}

\bibitem{abels}
H.~Abels.
\newblock Kompakt definierbare topologische {G}ruppen.
\newblock {\em Math. Ann.}, 197:221--233, 1972.

\bibitem{baderfurmangelandermonod}
U.~Bader, A.~Furman, T.~Gelander, and N.~Monod.
\newblock Property ({T}) and rigidity for actions on {B}anach spaces.
\newblock {\em Acta Math.}, 198(1):57--105, 2007.

\bibitem{badernowak}
U.~Bader and P.~W. Nowak.
\newblock Cohomology of deformations.
\newblock {\em J. Topol. Anal.}, 7(1):81--104, 2015.

\bibitem{baderrosendalsauer}
U.~Bader, C.~Rosendal, and R.~Sauer.
\newblock On the cohomology of weakly almost periodic group representations.
\newblock {\em J. Topol. Anal.}, 6(2):153--165, 2014.

\bibitem{bekkadelaharpevalette}
B.~Bekka, P.~de~la Harpe, and A.~Valette.
\newblock {\em Kazhdan's property ({T})}, volume~11 of {\em New Mathematical
  Monographs}.
\newblock Cambridge University Press, Cambridge, 2008.

\bibitem{brownFisherHurtado}
A.~Brown, D.~Fisher, S.~Hurtado.
\newblock Zimmer's conjecture: Subexponential growth, measure rigidity, and strong property ({T}).
\newblock arXiv:1608.04995 (2016).

\bibitem{CherixCowlingStraub}
P.-A. Cherix, M.~Cowling, and B.~Straub.
\newblock Filter products of {$C_0$}-semigroups and ultraproduct
  representations for {L}ie groups.
\newblock {\em J. Funct. Anal.}, 208(1):31--63, 2004.

\bibitem{dacunhacastellekrivine}
D.~Dacunha-Castelle and J.~L. Krivine.
\newblock Applications des ultraproduits \`a l'\'etude des espaces et des
  alg\`ebres de {B}anach.
\newblock {\em Studia Math.}, 41:315--334, 1972.


\bibitem{drutunowak}
C.~Dru\c{t}u and P.~W.~Nowak.
\newblock Kazhdan projections, random walks and ergodic theorems.
\newblock {\em J. Reine Angew. Math.}, 10.1515/crelle-2017-0002, 2017.

\bibitem{enflo}
P.~Enflo.
\newblock Banach spaces which can be given an equivalent uniformly convex norm.
\newblock {\em Israel J. Math.}, 13:281--288 (1973), 1972.

\bibitem{fishermargulis}
D.~Fisher and G.~Margulis.
\newblock Almost isometric actions, property ({T}), and local rigidity.
\newblock {\em Invent. Math.}, 162(1):19--80, 2005.

\bibitem{gromov}
M.~Gromov.
\newblock Random walk in random groups.
\newblock {\em Geom. Funct. Anal.}, 13(1):73--146, 2003.

\bibitem{Heinrich}
S.~Heinrich.
\newblock Ultraproducts in {B}anach space theory.
\newblock {\em J. Reine Angew. Math.}, 313:72--104, 1980.

\bibitem{delaatdelasalle1}
T.~de~Laat and M.~de~la Salle.
\newblock Strong property ({T}) for higher-rank simple {L}ie groups.
\newblock {\em Proc. Lond. Math. Soc. (3)}, 111(4):936--966, 2015.

\bibitem{delaatdelasalle2}
T.~de Laat and M.~de la Salle.
\newblock Approximation properties for noncommutative $L^p$-spaces of high rank lattices and nonembeddability of expanders.
\newblock {\em J. Reine Angew. Math.}, 10.1515/crelle-2015-0043, 2015.


\bibitem{lafforguestrongt}
V.~Lafforgue.
\newblock Un renforcement de la propri\'et\'e ({T}).
\newblock {\em Duke Math. J.}, 143(3):559--602, 2008.

\bibitem{lafforguefastfourier}
V.~Lafforgue.
\newblock Propri\'et\'e ({T}) renforc\'ee banachique et transformation de
  {F}ourier rapide.
\newblock {\em J. Topol. Anal.}, 1(3):191--206, 2009.

\bibitem{liao}
B.~Liao.
\newblock Strong {B}anach property ({T}) for simple algebraic groups of higher
  rank.
\newblock {\em J. Topol. Anal.}, 6(1):75--105, 2014.

\bibitem{JohnsonLindenstrauss}
W.~B. Johnson and J.~Lindenstrauss.
\newblock Basic concepts in the geometry of {B}anach spaces.
\newblock In {\em Handbook of the geometry of {B}anach spaces, {V}ol. {I}},
  pages 1--84. North-Holland, Amsterdam, 2001.

  
\bibitem{mimura}
M.~Mimura.
\newblock Metric Kazhdan constants.
\newblock {\em In preparation}.

\bibitem{oppenheim}
I.~Oppenheim.
\newblock Alternating projections, angles between groups and strengthening of property (T).
\newblock {\em Math. Ann.}, 367(1-2):623--666, 2017.

\bibitem{pisier}
G.~Pisier.
\newblock Martingales with values in uniformly convex spaces.
\newblock {\em Israel J. Math.}, 20(3-4):326--350, 1975.

\bibitem{raynaud}
Y.~Raynaud.
\newblock On ultrapowers of non commutative {$L_p$} spaces.
\newblock {\em J. Operator Theory}, 48(1):41--68, 2002.


\bibitem{delasalle1}
M.~de la Salle.
\newblock Towards Strong Banach property (T) for $\mathrm{SL}(3,\mathbb{R})$.
\newblock {\em Israel J.~Math} 211(1):105--145 (2015).


\bibitem{shalom}
Y.~Shalom.
\newblock Rigidity of commensurators and irreducible lattices.
\newblock {\em Invent. Math.}, 141(1):1--54, 2000.

\bibitem{shulman}
T.~ Shulman.
\newblock On subspaces of invariant vectors.
\newblock {\em Studia Math.}, 236(1):1--11, 2017.


\bibitem{stalder}
Y.~Stalder.
\newblock Fixed point properties in the space of marked groups.
\newblock In {\em Limits of graphs in group theory and computer science}, pages
  171--182. EPFL Press, Lausanne, 2009.

\bibitem{tzafriri}
L.~Tzafriri.
\newblock Remarks on contractive projections in {$L_{p}$}-spaces.
\newblock {\em Israel J. Math.}, 7:9--15, 1969.

\bibitem{yu}
G.~Yu.
\newblock Hyperbolic groups admit proper affine isometric actions on
  {$l^p$}-spaces.
\newblock {\em Geom. Funct. Anal.}, 15(5):1144--1151, 2005.

\bibitem{zimmer}
R.~J. Zimmer.
\newblock {\em Ergodic theory and semisimple groups}, volume~81 of {\em
  Monographs in Mathematics}.
\newblock Birkh\"auser Verlag, Basel, 1984.







\end{thebibliography}
\end{document}